\RequirePackage{fix-cm}
\documentclass[smallextended]{svjour3}       
\smartqed  
\usepackage{csquotes}
\usepackage{graphicx}
\usepackage{amsfonts}
\usepackage{amsmath}
\usepackage{amssymb}
\usepackage{geometry}
\usepackage{color}
\usepackage[hidelinks]{hyperref}
\usepackage{tikz}
\usetikzlibrary{positioning,arrows.meta}

\usepackage{placeins}

\newcommand{\B}[1]{\boldsymbol{#1}}

\usepackage{algorithm}      
\usepackage{algpseudocode}

\usepackage{cite}

\begin{document}

\title{Quadratic Weighted Histopolation on Tetrahedral Meshes with Probabilistic Degrees of Freedom
}

\titlerunning{Quadratic Weighted Histopolation on Tetrahedral Meshes with Probabilistic Degrees of Freedom}        

\author{ Allal Guessab \and 
        Federico Nudo 
}


\institute{ 
           Allal Guessab \at
            Laboratoire de Mathématiques et de leurs Applications, UMR CNRS 5142, Université de Pau et des Pays de l'Adour (UPPA), 64000 Pau, France  \\
    \email{allal.guessab@univ-pau.fr}
               \and 
            Federico Nudo (corresponding author) \at
              Department of Mathematics and Computer Science, University of Calabria, Rende (CS), Italy \\
    \email{federico.nudo@unical.it}
}

\date{Version: September 20, 2025}

\maketitle

\begin{abstract}
In this paper we introduce three complementary three-dimensional weighted quadratic enrichment strategies to improve the accuracy of local histopolation on tetrahedral meshes. The first combines face and interior weighted moments (face–volume strategy), the second uses only volumetric quadratic moments (purely volumetric strategy), and the third enriches the quadratic space through edge-supported probabilistic moments (edge–face strategy). All constructions are based on integral functionals defined by suitable probability densities and orthogonal polynomials within quadratic trial spaces. We provide a comprehensive analysis that establishes unisolvence and derives necessary and sufficient conditions on the densities to guarantee well-posedness. Representative density families—including two-parameter symmetric Dirichlet laws and convexly blended volumetric families—are examined in detail, and a general procedure for constructing the associated quadratic basis functions is outlined. For all admissible densities, an adaptive algorithm automatically selects optimal parameters. Extensive numerical experiments confirm that the proposed strategies yield substantial accuracy improvements over the classical linear histopolation scheme.
\end{abstract}


\section{Introduction}
Reconstructing functions from indirect data is a fundamental challenge in applied mathematics and scientific computing. Applications range from imaging to numerical simulation.
Often direct pointwise measurements of the target function are unavailable, and only averages or integrals over prescribed geometric regions can be accessed.
Classical examples include tomography, where detectors register line integrals of the attenuation
coefficient of a medium, and ultrasound or seismic imaging~\cite{kak2001principles,natterer2001mathematics,palamodov2016reconstruction}.
A natural computational paradigm for such problems is offered by
\emph{histopolation methods}.
Unlike classical interpolation, which matches pointwise samples, histopolation relies on integral data, seeking an approximant whose integrals against suitable test functions or densities reproduce the available information.
This perspective is especially appealing when conservation principles or natural averaging are essential, providing a bridge between approximation theory and weighted reconstruction techniques.
In recent years, local histopolation methods have been extensively investigated;
see, e.g.,~\cite{nudo1,nudo2,canwapaper}.
Formally, a local histopolation method can be described by the triple~\cite{Guessab:2022:SAB}
\begin{equation*}
    \left(S_d, \mathbb{F}_{S_d}, \Sigma_{S_d}\right),
\end{equation*}
where  
\begin{itemize}
  \item $S_d$ is a polytope in $\mathbb{R}^d$, $d \ge 1$;
  \item $\mathbb{F}_{S_d}$ is an $n$-dimensional trial space on $S_d$;
  \item $\Sigma_{S_d}=\left\{\mathcal{L}_j: j=1,\dots,n\right\}$ is a set of linearly independent weighted integral functionals called the \emph{degrees of freedom},
\end{itemize}
such that the triple is \emph{unisolvent}, meaning that any $f \in \mathbb{F}_{S_d}$ with $\mathcal{L}_j(f)=0$, $j=1,\dots,n$, 
is identically zero.
A set of functions $\left\{\varphi_i : i=1,\dots,n\right\}$
forming a basis of $\mathbb{F}_{S_d}$ and satisfying
\[
\mathcal{L}_j\left(\varphi_i\right)
=\delta_{ij}
=\begin{cases}
1, & \text{if } i=j,\\[2mm]
0, & \text{otherwise},
\end{cases}
\]
constitutes the basis associated with the local histopolation method.
The idea of such methods is to partition the computational domain into subregions,
endow each of them with an appropriate approximation space,
and assemble a global reconstruction by patching these local approximations together.
If the resulting approximation is discontinuous across the interfaces of the subdomains,
the histopolation method is said to be \emph{nonconforming}; otherwise it is \emph{conforming}. The classical nonconforming linear histopolation method can be locally represented by
\begin{equation}\label{CRelement}
\mathcal{CH} = \left(T,\mathbb{P}_1(T),\Sigma^{\mathrm{CH}}\right),
\end{equation}
where 
\begin{itemize}
    \item $T$ is a nondegenerate tetrahedron in $\mathbb{R}^3$ with vertices $\B{v}_i$, $i=1,2,3,4$;
    \item $\mathbb{P}_1(T)$ denotes the space of linear polynomials on $T$;
    \item the set of degrees of freedom is 
    \begin{equation*}
    \Sigma^{\mathrm{CH}}=\left\{\mathcal{I}_j^{\mathrm{CH}}:f \in C(T)\mapsto
   \frac{1}{\left|F_j\right|}\int_{F_j} f(\B{x})d\B{x}\in\mathbb{R} \, :\, j=1,2,3,4\right\},
    \end{equation*}
    where $F_j$ denotes the face of $T$ opposite to the vertex $\B{v}_j$.
\end{itemize}
The basis functions associated with $\mathcal{CH}$, expressed in barycentric coordinates, are
\[
\varphi_i^{\mathrm{CH}} = 1 - 3\lambda_i,\quad i=1,2,3,4,
\]
and the corresponding reconstruction operator is
\begin{equation*}
   \pi_1^{\mathrm{CH}} : f \in C(T) \mapsto \displaystyle\sum_{j=1}^{4}\mathcal{I}^{\mathrm{CH}}_{j}(f)\,\varphi^{\mathrm{CH}}_{j} \in \mathbb{P}_1(T).
\end{equation*}
In recent years, histopolation has attracted increasing interest, leading to significant advances in both theory and applications~\cite{BruniErbFekete,bruni2024polynomial,dell2025c}.
Notable developments include global polynomial histopolation–regression frameworks in one and several variables~\cite{bruni2025polynomial} and weighted local histopolation schemes in the bivariate case~\cite{nudo2}, motivated by applications in image reconstruction, signal processing, and the design of structure-preserving numerical methods~\cite{Bosner:2020:AOC,HiemstraJCP}.
Despite these advances, most existing local enriched weighted histopolation methods remain confined to one or two dimensional settings.
While histopolation has been extensively studied in lower dimensions, our focus here is on three-dimensional tetrahedra. Nevertheless, the theoretical framework naturally extends to $d$-dimensional simplices, since all constructions rely only on barycentric coordinates and orthogonality arguments.
This limitation is especially critical for tomography and wave propagation, where accurate three-dimensional reconstructions are essential.

In this work we develop a comprehensive framework for \emph{probabilistic quadratic local weighted histopolation methods on tetrahedra}. We enrich the classical nonconforming linear scheme with weighted integral functionals based on suitable probability densities and orthogonal polynomials, combined with quadratic trial polynomials. Specifically, we introduce three complementary enrichment strategies:
one coupling face and interior weighted moments (face-volume strategy),
one relying exclusively on purely volumetric quadratic moments (purely-volumetric strategy),
and one enriching the quadratic component through edge-supported probabilistic moments (edge-face strategy).
We show how to construct the quadratic basis functions associated with these new weighted enriched histopolation methods.
Special attention is devoted to representative families of probability densities,
including the two-parameter symmetric Dirichlet family for the face-volume case
and convexly blended volumetric families.
For each family we employ an adaptive procedure to select the parameters
that minimize the reconstruction error.

The main contributions of this work can be summarized as follows:
\begin{itemize}
  \item We propose and analyze three complementary quadratic enrichment strategies for local weighted histopolation on tetrahedra:
  \begin{itemize}
      \item  a face-volume strategy combining face and interior weighted moments,
      \item a purely-volumetric strategy based on interior quadratic moments,
      \item an edge-face strategy employing edge-supported probabilistic moments.
  \end{itemize}
  \item We establish a \emph{general unisolvence theorem} and provide necessary and sufficient conditions on the probability densities that guarantee well-posedness of all three enriched schemes.
  \item We present a \emph{constructive procedure} for deriving the quadratic basis functions associated with probabilistic degrees of freedom.
  \item For the face-volume strategy we give a detailed realization within the two-parameter Dirichlet family, proving unisolvence for every $\alpha,\beta>0$,
        and we supply explicit formulas for the associated matrices and their determinants.
  \item For the purely-volumetric strategy we analyze Dirichlet interior densities and their convex blends $\Omega^{(\theta,\gamma)}$, proving unisolvence and symmetric positive definiteness of the quadratic moment matrices for all $\theta\in[0,1]$ and $\gamma>0$.
  \item For the edge-face strategy we formulate suitable Beta-type edge densities, derive explicit unisolvence conditions, and discuss their computational advantages for edge-aligned features.
\end{itemize}

\medskip
The paper is organized as follows. In Section~\ref{sec2} we introduce three 3D quadratic enrichment strategies and establish a general unisolvence theorem together with representative families of probability densities ensuring unisolvence. A general procedure for deriving explicit formulas for the associated quadratic basis functions is also provided. Section~\ref{ex3} analyzes Dirichlet densities: for the face-volume strategy we consider the two-parameter symmetric family with $\alpha,\beta>0$; for the purely-volumetric strategy we examine the Dirichlet interior density and its convex blends, and discuss their symmetric positive definiteness and implications for unisolvence and accuracy; and for the edge-face strategy we study Beta-type edge densities and verify the corresponding unisolvence conditions. In Section~\ref{sec3}, we describe an adaptive algorithm for the optimal selection of the density parameters. Finally, in Section~\ref{sec4}, we present numerical experiments demonstrating the superior accuracy of the proposed quadratic enriched methods compared with the classical linear histopolation scheme.

\section{Quadratic Enriched Weighted Histopolation Method on a Tetrahedron}\label{sec2}
Let $T \subset \mathbb{R}^3$ be a nondegenerate tetrahedron with vertices 
$\B{v}_i$, $i=1,2,3,4$. 
We denote by $\lambda_i$, $i=1,2,3,4$, the barycentric coordinates on $T$, i.e., the unique affine functions such that
\begin{equation}\label{propisp} 
\sum_{i=1}^4 \lambda_i = 1,
\quad \lambda_i\left(\B v_j\right) = \delta_{ij}.
\end{equation}
Let $\mathbb{P}_2(T)=\operatorname{span}\left\{\lambda_l, \lambda_i \lambda_j \,: \, l=1,2,3,4, \, 1\le i < j \le 4\right\} $
be the space of polynomials on $T$ of total degree at most two.
It is well known that $\dim \mathbb{P}_2(T) = \binom{3+2}{2} = 10. $
In the discussion below we will use a well-known result,
which holds for any $d$-simplex $S_d$, $d \in \mathbb{N}$~\cite{szego1975orthogonal}.
\begin{lemma}\label{lem1}
  Let $\alpha_0,\ldots, \alpha_d>-1$. Then the following identity holds
  \begin{equation}\label{idbc}
     \frac{1}{\left|S_d\right|}\int_{S_d}\prod_{i=0}^{d}\lambda_i^{\alpha_i}(\B x) d{\B x}=\frac{d! \prod_{i=0}^{d}\Gamma\left(\alpha_i+1\right) }{\Gamma\left(d+1+\sum_{i=0}^{d}\alpha_i\right)},
  \end{equation}
  where $\left \lvert S_d\right \rvert$ is the volume of $S_d$ and $\Gamma(z)$ is the gamma function~\cite{Abramowitz:1948:HOM}.
\end{lemma}

\subsection{Face–volume enrichment strategy}
We now introduce a quadratic enriched weighted histopolation method that combines face and interior information. To this end, we define a set of linear functionals on $\mathbb{P}_2(T)$.
\begin{itemize}
    \item For each face $F_j$ opposite to vertex $\B{v}_j$, fix a probability density $\omega_j$ 
        on $F_j$ and choose a nonzero quadratic polynomial
        $q_j$ on $F_j$ such that
  \begin{equation}\label{cond11}
     \left\langle q_j,\phi \right\rangle_{F_j,\omega_j}=\int_{F_j} q_j(\B x) \phi(\B x) \omega_j (\B x) d\B x = 0, 
  \quad \forall \ \text{affine }\phi \text{ on }F_j.
  \end{equation}
 We then define
\begin{eqnarray}\label{Ij}
\mathcal{I}_{j} : p \in \mathbb{P}_2(T) \mapsto \int_{F_j} 
p(\B x)\omega_j(\B x) d\B x \in \mathbb R,
\quad j=1,2,3,4,\\ \label{Lj}
\mathcal{L}_{j} : p \in \mathbb{P}_2(T) \mapsto \left\langle p, q_j\right\rangle_{F_j,\omega_j} \in \mathbb R,
\quad j=1,2,3,4. 
\end{eqnarray}
\item To complete the set of degrees of freedom, choose a symmetric probability
        density $\Omega$ on $T$ and two distinct quadratic polynomials
        $\rho_1,\rho_2 \in \mathbb{P}_2(T)$ that are orthogonal to all affine
        functions on $T$ with respect to the inner product
  \begin{equation}\label{innprd}
      \left\langle f,g\right\rangle_{T,\Omega}=\int_T f(\B x) g(\B x) \Omega(\B x) d\B x,
  \end{equation}
that is  $ \left\langle \rho_i,\phi\right\rangle_{T,\Omega}=0$, for any $\phi\in\mathbb{P}_1(T)$.  Then we define
\begin{equation}\label{Vk}
    \mathcal{V}_{k} : p \in \mathbb{P}_2(T) \mapsto \left\langle \rho_k,p\right\rangle_{T,\Omega} \in \mathbb R,
\quad k=1,2.
\end{equation}
\end{itemize}
We consider the triple
\begin{equation}\label{AFfv}
\mathcal{AF}^{\mathrm{fv}}=\left(T,\mathbb{P}_2(T),\Sigma^{\mathrm{fv}}\right), 
\end{equation}
where $\Sigma^{\mathrm{fv}} = \left\{ \mathcal{I}_j, \mathcal{L}_j, \mathcal{V}_k \, :\,   j=1,2,3,4, \quad k=1,2 \right\}$.  
To prove the unisolvence of this triple we introduce the linear map
\begin{equation}\label{funD}
   \mathcal{D}^{\mathrm{fv}} : p \in \mathbb{S}_2(T) \mapsto \left[\mathcal{L}_1(p),\dots,\mathcal{L}_4(p),\mathcal{V}_1(p),\mathcal{V}_2(p)\right]^{\top}\in\mathbb{R}^6,
\end{equation}
defined on 
\begin{equation}\label{S2}
    \mathbb S_2(T)=\operatorname{span}\left\{\lambda_i\lambda_j\, : \,1\le i<j\le 4\right\},
\end{equation}
with basis
\begin{equation}\label{basisfun}
    \mathrm{B}=\left\{b_1,\dots,b_6\right\}=\left\{\lambda_1\lambda_2, \lambda_1\lambda_3, \lambda_1\lambda_4,
\lambda_2\lambda_3, \lambda_2\lambda_4, \lambda_3\lambda_4\right\}.
\end{equation}
The following lemma is crucial.

\begin{lemma}\label{lem:invert-M}
The triple $ \left(T, \mathbb{P}_1(T), \left\{\mathcal{I}_j\, :\, j=1,2,3,4\right\}\right)$
is unisolvent. 
\end{lemma}
\begin{proof}
Let $a\in\mathbb{P}_1(T)$ be an affine function such that
\begin{equation}\label{sasasasax1}
    \mathcal{I}_j(a)=0,\quad j=1,2,3,4.
\end{equation}
By the one–point Gaussian quadrature (the face–centroid rule), we have
\begin{equation}\label{sasasasax}
\mathcal{I}_j(a)=W_ja\left(\B\xi_j\right), \quad W_j>0,
\end{equation}
where $\B{\xi}_j$ is the quadrature point in the interior of the face $F_j$.  
From~\eqref{sasasasax1} and~\eqref{sasasasax} it follows that $a\left(\B \xi_j\right)=0$, $j=1,2,3,4$. Since the points $\B{\xi}_j$, $j=1,2,3,4$, are not coplanar, the zero set of a nonzero affine function in $\mathbb{R}^3$, which is necessarily a plane, cannot contain them all~\cite{Ciarlet2002TFE}. 
Hence $a= 0$. 
This proves the unisolvence of the triple.
\end{proof}

The following result provides an equivalent characterization of unisolvence for the enriched quadratic triple via a rank condition on its quadratic component.

\begin{theorem}[Unisolvence Equivalence - Rank Condition]\label{thrmimps}
The enriched triple $\mathcal{AF}^{\mathrm{fv}}$ is unisolvent if and only if the linear operator $\mathcal{D}^{\mathrm{fv}}$ defined in~\eqref{funD} satisfies
\begin{equation}\label{funDcond}
\operatorname{rank}\left(\mathcal{D}^{\mathrm{fv}}\right)=6.
\end{equation}
\end{theorem}
\begin{proof}
We first assume that~\eqref{funDcond} holds. Let $p \in \mathbb{P}_2(T)$ satisfy
\begin{eqnarray}\label{vancond}
    \mathcal{I}_j(p)&=&0,  \quad j=1,2,3,4, \\ \label{vancond1} \mathcal{L}_j(p)&=&0,  \quad j=1,2,3,4,\\ \mathcal{V}_k(p)&=&0, \quad k=1,2. \label{vancond2}
\end{eqnarray}
We show that $p= 0$. By construction every $p\in\mathbb{P}_2(T)$ can be uniquely decomposed as
\begin{equation*}
p= a + q, \quad 
 a \in \mathbb{P}_1(T),\quad q \in \mathbb{S}_2(T),     
\end{equation*}  
where $\mathbb{S}_2(T)$ is defined in~\eqref{S2}.
Using \eqref{Lj} and the fact that each $q_j$ is orthogonal to all affine functions on $F_j$,  
conditions~\eqref{vancond1} reduce to
\begin{equation}\label{new}
0=\mathcal{L}_j(p)=\mathcal{L}_j(q), \quad j=1,2,3,4.
\end{equation}
Similarly, since $\rho_1,\rho_2$ are orthogonal to all affine functions on $T$,
conditions~\eqref{vancond2} reduce to
\begin{equation}\label{newnew}
    0=\mathcal{V}_k(p)=\mathcal{V}_k(q), \quad k=1,2.
\end{equation}
Together,~\eqref{new} and~\eqref{newnew} yield $ \mathcal{D}^{\mathrm{fv}}(q)=0$. Since, by assumption, $\mathcal{D}^{\mathrm{fv}}$ has rank $6$, it follows that $q = 0$.
Hence $p=a$ is affine.  
From~\eqref{vancond} and Lemma~\ref{lem:invert-M} it follows that $a= 0$, and therefore $p= 0$.

Conversely, suppose that $\mathcal{AF}^{\mathrm{fv}}$ is unisolvent and assume, to reach a contradiction, that $\operatorname{rank}\left(\mathcal{D}^{\mathrm{fv}}\right)<6.$ Then there exists a nonzero polynomial $\tilde{q}\in\mathbb{S}_2(T)$ such that $\tilde{q}\in \ker\left(\mathcal{D}^{\mathrm{fv}}\right)$.
By Lemma~\ref{lem:invert-M} there is a unique affine function $\tilde{a}$ such that $ \mathcal{I}_j\left(\tilde{a}\right)=-\mathcal{I}_j\left(\tilde{q}\right).$
We set $\tilde{p}=\tilde{a}+\tilde{q}\in\mathbb{P}_2(T)$. 
Then
\[
\mathcal{I}_j\left(\tilde{p}\right)=\mathcal{I}_j\left(\tilde{a}\right)+\mathcal{I}_j\left(\tilde{q}\right)=0, \quad j=1,2,3,4,
\]
and, by construction of the additional functionals, $\mathcal{L}_j\left(\tilde{a}\right)=\mathcal{V}_k\left(\tilde{a}\right)=0$ while, by assumption, $\mathcal{L}_j\left(\tilde{q}\right)=\mathcal{V}_k\left(\tilde{q}\right)=0$. 
Thus
$$\mathcal{L}_j\left(\tilde{p}\right)=\mathcal{V}_k\left(\tilde{p}\right)=0, \quad j=1,2,3,4, \quad k=1,2.$$  
Hence $\tilde{p}\in\mathbb{P}_2(T)$ is nonzero and all degrees of freedom in $\Sigma^{\mathrm{fv}}$ vanish, contradicting unisolvence.
Therefore $\operatorname{rank}\left(\mathcal{D}^{\mathrm{fv}}\right)=6$.
\end{proof}

\begin{remark}
The previous theorem shows that the enriched triple $\mathcal{AF}^{\mathrm{fv}}$
is unisolvent if and only if $\operatorname{rank}\left(\mathcal{D}^{\mathrm{fv}}\right)=6$.
Equivalently, unisolvence of the full set of degrees of freedom $\Sigma^{\mathrm{fv}}$ on $\mathbb{P}_2(T)$ amounts to unisolvence of the reduced set
\[
\left\{\mathcal{L}_j,\mathcal{V}_k : j=1,\dots,4,\; k=1,2\right\}
\]
on the pure quadratic subspace $\mathbb{S}_2(T)$.  Indeed, the face averages $\mathcal{I}_j$ already ensure unisolvence on the affine
part $\mathbb{P}_1(T)$, see Lemma~\ref{lem:invert-M}.
Therefore establishing unisolvence for the full space $\mathbb{P}_2(T)$
reduces to verifying unisolvence of $\left\{\mathcal{L}_j,\mathcal{V}_k\right\}$ on $\mathbb{S}_2(T)$.
\end{remark}

The face-volume strategy is a natural enrichment of classical histopolation, combining surface and volumetric information to better capture both boundary and interior features of the target function. In what follows we present two representative families of admissible probability densities for which the associated quadratic enriched triple is unisolvent.
\begin{example}[Uniform densities]\label{ex1}
We consider the uniform densities
\begin{equation*}
    \omega_j= \frac{1}{\left\lvert F_j\right\rvert}, \quad \Omega=\frac{1}{\left\lvert T\right\rvert}, \quad j=1,2,3,4,
\end{equation*}
where $\left\lvert F_j\right\rvert$ is the area of the face $F_j$ and $\left\lvert T\right\rvert$ is the volume of $T$. On each face $F_j$, let
$\mu_{1,j},$ $\mu_{2,j},$ $\mu_{3,j}$ be the barycentric coordinates associated with its three vertices, and set
\begin{equation*}
q_j=\mu_{1,j}^2+\mu_{2,j}^2+\mu_{3,j}^2-\frac{1}{2},\quad j=1,2,3,4.    
\end{equation*}
Using the moment formula~\eqref{idbc}, a direct computation shows
\begin{equation*}
    \frac{1}{\left\lvert F_j\right\rvert}\int_{F_j} q_j(\B x)d\B x=0,\quad \frac{1}{\left\lvert F_j\right\rvert}
\int_{F_j} q_j(\B x)\mu_{r,j}(\B x)d\B x=0,\quad r=1,2,3,
\end{equation*}
so $q_j$ is orthogonal to all affine functions on $F_j$ with respect to the
uniform surface measure. 

\noindent
Now, we set
\begin{equation*}
    \rho_1 = \lambda_1\lambda_2 + \frac{1}{30}-\frac{1}{6}\left(\lambda_1+\lambda_2\right),\quad
\rho_2= \lambda_1\lambda_3 + \frac{1}{30}-\frac{1}{6}\left(\lambda_1+\lambda_3\right).
\end{equation*}
Again, using~\eqref{idbc} a direct computation shows
\begin{equation*}
    \frac{1}{\left\lvert T\right\rvert}\int_T \rho_k(\B x)d\B x=0,\quad
 \frac{1}{\left\lvert T\right\rvert}\int_T \rho_k(\B x)\lambda_r(\B x)d\B x=0,\quad r=1,2,3,\quad k=1,2,
\end{equation*}
so $\rho_1,\rho_2$ are orthogonal to all affine functions on $T$ with respect to the uniform density $\Omega$.

With the basis $\mathrm{B}$ of $\mathbb{S}_2(T)$ defined in~\eqref{basisfun}, 
the matrix of $\mathcal{D}^{\mathrm{fv}}$ in~\eqref{funD} relative to $\mathrm{B}$ is
\begin{equation*}
\mathcal D^{\mathrm{fv}}=\frac{1}{720}\begin{pmatrix}
0 & 0 & 0 & -2 & -2 & -2\\
0 & -2 & -2 & 0 & 0 & -2\\
-2 & 0 & -2 & 0 & -2 & 0\\
-2 & -2 & 0 & -2 & 0 & 0\\
\frac{22}{35} & -\frac{3}{35} & -\frac{3}{35} & -\frac{3}{35} & -\frac{3}{35} & \frac{2}{35}\\
-\frac{3}{35} & \frac{22}{35} & -\frac{3}{35} & -\frac{3}{35} & \frac{2}{35} & -\frac{3}{35}
\end{pmatrix}.
\end{equation*}
Its determinant is
\begin{equation*}
    \det\left(\mathcal D^{\mathrm{fv}}\right)=\frac{1}{ 2\left(3^2(12)^8(5)^6(7)^2\right)}\neq 0;
\end{equation*}
so $\operatorname{rank}\left(\mathcal D^{\mathrm{fv}}\right)=6$.
\end{example}

\begin{example}[Symmetric quadratic densities]
We consider the symmetric quadratic densities
\begin{equation*}
    \omega_j= \frac{2}{\lvert F_j\rvert}\left(\mu_{1,j}^2+\mu_{2,j}^2+\mu_{3,j}^2\right), 
    \quad 
    \Omega= \frac{5}{2\lvert T\rvert}\left(\lambda_1^2+\lambda_2^2+\lambda_3^2+\lambda_4^2\right), \quad j=1,2,3,4,
\end{equation*}
see Fig.~\ref{fig:densitiesex2}. 
We set
\begin{equation*}
q_j = \mu_{1,j}^2+\mu_{2,j}^2+\mu_{3,j}^2-\frac{8}{15}, \quad j=1,2,3,4.
\end{equation*}
Using the moment formula~\eqref{idbc} with respect to the weighted measure $\omega_j$, we obtain
\begin{equation*}
\int_{F_j} q_j(\B x)\omega_j(\B x)d\B x=0,\quad
\int_{F_j} q_j(\B x)\mu_{r,j}(\B x)\omega_j(\B x)d\B x=0, \quad r=1,2,3.
\end{equation*}
so each $q_j$ is orthogonal to all affine functions on $F_j$ with respect to the symmetric quadratic surface measure.

\noindent
Now, we consider the following polynomials 
\begin{equation*}
    \rho_1 = \lambda_1\lambda_2 + \frac{23}{840} - \frac{3}{20}\left(\lambda_1+\lambda_2\right),\quad
    \rho_2 = \lambda_1\lambda_3 + \frac{23}{840} - \frac{3}{20}\left(\lambda_1+\lambda_3\right).
\end{equation*}
Again, using~\eqref{idbc} one finds
\begin{equation*}
\frac{1}{\lvert T\rvert}\int_T \rho_k(\B x)\Omega(\B x)d\B x=0,\quad
\frac{1}{\lvert T\rvert}\int_T \rho_k(\B x)\lambda_r(\B x)\Omega(\B x)d\B x=0,\quad r=1,2,3,4,\  k=1,2,
\end{equation*}
hence $\rho_1,\rho_2$ are orthogonal to all affine functions on $T$ with respect to $\Omega$.

With the basis $\mathrm{B}$ of $\mathbb{S}_2(T)$ defined in~\eqref{basisfun}, 
the matrix of $\mathcal{D}^{\mathrm{fv}}$ in~\eqref{funD} is
\begin{equation*}
    \mathcal D^{\mathrm{fv}} = \frac{1}{ 2116800}
\begin{bmatrix}
0 & 0 & 0 & -7168 & -7168 & -7168\\
0 & -7168 & -7168 & 0 & 0 & -7168\\
-7168 & 0 & -7168 & 0 & -7168 & 0\\
-7168 & -7168 & 0 & -7168 & 0 & 0\\
2165 & -250 & -250 & -250 & -250 & 135\\
-250 & 2165 & -250 & -250 & 135 & -250
\end{bmatrix},
\end{equation*}
which has a nonzero determinant and therefore $\operatorname{rank}\left(\mathcal D^{\mathrm{fv}}\right)=6.$
\end{example}

\begin{figure}
      \centering
\includegraphics[width=0.49\linewidth]{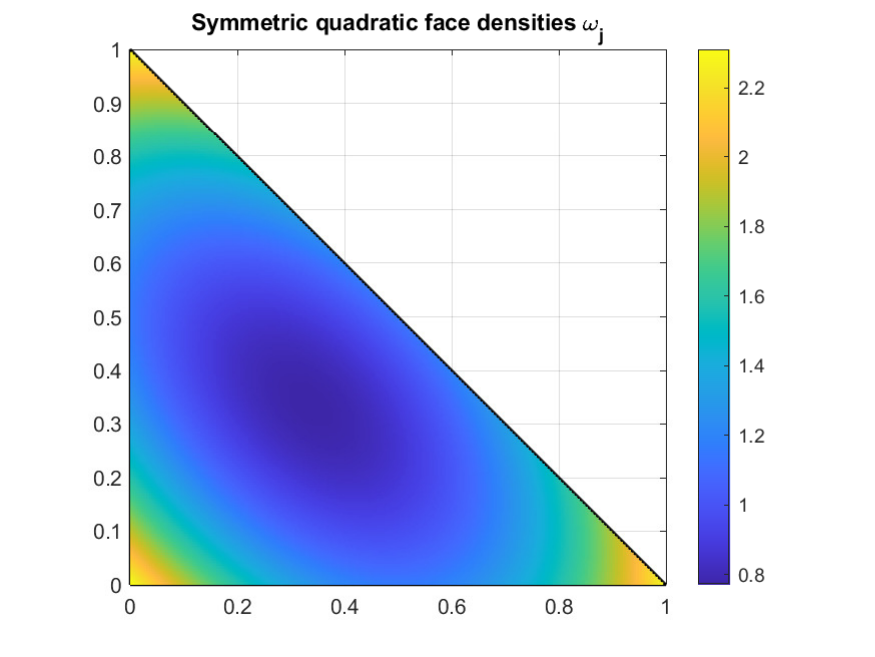}
\caption{Symmetric quadratic face densities $\omega_j$ on the reference tetrahedron.}
      \label{fig:densitiesex2}
  \end{figure}

\medskip

\noindent
In the following we define

\begin{equation}\label{pi1ab}
   \mathcal{H}:p\in\mathbb{P}_{2}(T) \mapsto \displaystyle\left[\mathcal{I}_{1}(p),\dots,\mathcal{I}_{4}(p),\mathcal{L}_{1}(p), \dots,\mathcal{L}_{4}(p),\mathcal{V}_{1}(p),\mathcal{V}_{2}(p)\right]^{\top} \in \mathbb{R}^{10}.
\end{equation}
Let
\begin{equation*}
\mathrm{B}^{\mathrm{compl}}=\left\{\lambda_1,\lambda_2,\lambda_3,\lambda_4, \lambda_1\lambda_2, \lambda_1\lambda_3, \lambda_1\lambda_4,
\lambda_2\lambda_3, \lambda_2\lambda_4, \lambda_3\lambda_4\right\}.
\end{equation*}
be the chosen basis of $\mathbb{P}_2(T)$.
With respect to $\mathrm{B}^{\mathrm{compl}}$, the matrix of $\mathcal{H}$ has the block form
\[
\mathcal{H}= \left[ \begin{array}{c|c} \mathcal{M} & \mathcal{N}\\ \hline \B 0 & \mathcal{D}^{\mathrm{fv}} \end{array} \right]\in\mathbb{R}^{10\times10}, 
\]
where $\mathcal{M}\in\mathbb{R}^{4\times 4}$,  $\mathcal{N}\in\mathbb{R}^{4\times 6}$, and $\mathcal{D}^{\mathrm{fv}}\in\mathbb{R}^{6\times 6}$.
We write its inverse column-wise as
\begin{equation*}
\tilde{\mathcal{H}}=\mathcal{H}^{-1}=\left[\tilde{\boldsymbol{h}}_{1},\dots,\tilde{\boldsymbol{h}}_{10}\right]\in\mathbb{R}^{10\times 10}.
\end{equation*}

\begin{theorem}\label{theoimp_ab}
The basis functions of the enriched triple $\mathcal{AF}^{\mathrm{fv}}$ can be written as
\begin{equation}\label{genbasis_ab}
\chi_{\ell}=\left\langle \tilde{\B{h}}_{\ell}, \B{\Lambda}\right\rangle, \quad \ell=1,\dots,10,
\end{equation}
\noindent
where $\B{\Lambda}=\left[\Lambda_1,\dots,\Lambda_{10}\right]=\left[\lambda_{1},\lambda_2,\lambda_3,\lambda_{4},\lambda_1\lambda_2, \lambda_1\lambda_3, \lambda_1\lambda_4,
\lambda_2\lambda_3, \lambda_2\lambda_4, \lambda_3\lambda_4\right]^\top.$
\end{theorem}
\begin{proof}
Since $\chi_\ell\in\mathbb{P}_2(T)$, there exists a unique vector 
$\B{c}^{(\ell)}=\left[c^{(\ell)}_{1},\dots,c^{(\ell)}_{10}\right]^{\top}\in\mathbb{R}^{10}$ such that
\begin{equation}\label{eq:phi-expansion-ab}
\chi_{\ell}(\B x)=\sum_{i=1}^{10}c^{(\ell)}_{i}\Lambda_{i}(\B x)
=\left\langle \B{c}^{(\ell)},\B{\Lambda}(\B x)\right\rangle .
\end{equation}
The basis functions satisfy the Kronecker conditions
$\mathcal{H}\left(\chi_\ell\right)=\B{e}_\ell$, where $\B{e}_\ell$ is the $\ell$-th canonical vector of $\mathbb{R}^{10}$.
Using the linearity of $\mathcal{H}$ and~\eqref{eq:phi-expansion-ab} gives
\[
\B{e}_{\ell}
=\mathcal{H}\left(\chi_{\ell}\right)
=\mathcal{H}\left(\sum_{i=1}^{10}c^{(\ell)}_{i}\Lambda_{i}\right)
=\mathcal{H}\B{c}^{(\ell)}.
\]
 Hence $\B{c}^{(\ell)}
=\tilde{\B h}_{\ell}.$ Substituting this identity into~\eqref{eq:phi-expansion-ab} yields~\eqref{genbasis_ab}.
\end{proof}
We define the reconstruction operator associated with the triple
$\mathcal{AF}^{\mathrm{fv}}$ introduced in~\eqref{AFfv} by

\begin{equation}\label{pi12}
   {\pi}_{2}^{{\mathrm{fv}}} : f \in C(T) \mapsto \displaystyle\displaystyle{\sum_{j=1}^{4}  {\mathcal{I}}_{j}\left(f\right)\chi_{j}+ \sum_{j=1}^{4}  {\mathcal{L}}_{j}\left(f\right){\chi}_{j+4} + \sum_{k=1}^{2}  {\mathcal{V}}_{k}\left(f\right){ \chi}_{k+8}} \in \mathbb{P}_2(T).
\end{equation}

\subsection{Purely–volumetric enrichment strategy}\label{sec:volume-enrichment} 
A natural variant of the previous enrichment strategy retains the face averages
while replacing all higher-order degrees of freedom with six purely volumetric
quadratic moments supported inside $T$.
In this purely volumetric setting the functionals $\mathcal{I}_j$ keep their
classical definition as uniform surface averages, whereas only the quadratic
moments involve the blended interior density.

Let $\Omega$ be an arbitrary probability density on $T$.  
We consider six quadratic polynomials $\psi_{ij}\in\mathbb{P}_2(T)$, $1\le i<j\le4$, 
satisfying the orthogonality conditions
\begin{equation*}
   \left\langle \psi_{ij},\varphi\right\rangle_{T,\Omega}=  0,
\quad \forall \varphi\in \mathbb{P}_{1}(T),\quad 1\le i<j\le4,
\end{equation*}
where $\left\langle\cdot,\cdot\right\rangle_{T,\Omega}$ is defined in~\eqref{innprd}.
The associated volumetric functionals are 
\begin{equation*}
     \mathcal{F}_{ij} : p \in \mathbb{P}_2(T) \mapsto \int_{T}\psi_{ij}(\B x)p(\B x)\Omega(\B x)d\B x,
\quad 1\le i<j\le4.\\ 
\end{equation*}
The corresponding quadratic histopolation triple is
\begin{equation*}
\mathcal{AF}^{\mathrm{vol}}
   =\left(T,\mathbb{P}_{2}(T),\Sigma^{\mathrm{vol}}\right),
\end{equation*}
where the set of degrees of freedom is
\begin{equation*}
    \Sigma^{\mathrm{vol}}
   =\left\{\mathcal{I}_{j} \, :\,  j=1,2,3,4\right\}\cup\left\{\mathcal{F}_{ij} \, : \, 1\le i<j\le4\right\}.
\end{equation*}
Its restriction to the pure quadratic subspace $\mathbb{S}_2(T)$ is represented by the symmetric $6\times6$ matrix associated with
\begin{equation*}
   \mathcal{D}^{\mathrm{vol}} : p \in \mathbb{S}_2(T) \mapsto \left[\mathcal{F}_{12}(p),\mathcal{F}_{13}(p), \mathcal{F}_{14}(p),\mathcal{F}_{23}(p),\mathcal{F}_{24}(p),\mathcal{F}_{34}(p)\right]^{\top}\in\mathbb{R}^6,
\end{equation*}
whose rank determines the unisolvence of the enriched triple.
The following theorem provides a precise characterization.
\begin{theorem}\label{thrmimps1}
The enriched triple $\mathcal{AF}^{\mathrm{vol}}$ is unisolvent if and only if $$\operatorname{rank}\left(\mathcal{D}^{\mathrm{vol}}\right)=6.$$
\end{theorem}
\begin{proof}
The proof is identical to that of Theorem~\ref{thrmimps},
with the face-volume functionals replaced by the purely volumetric moments
$\mathcal{F}_{ij}$.
\end{proof}

In analogy with~\eqref{pi12}, we can therefore define the quadratic reconstruction operator
$\pi^{\mathrm{vol}}_{2}$ associated with the volumetric enriched triple.
Unlike the previous construction, which combines face and interior moments and requires
separate face and volume densities, the present strategy
\begin{itemize}
  \item relies only on a single interior probability density $\Omega$, independent of any choice of face weights;
  \item replaces the four quadratic face moments $\mathcal{L}_{j}$ and the two interior
        moments $\mathcal{V}_{k}$ by six purely volumetric moments $\mathcal{F}_{ij}$,
        while preserving the total number of ten degrees of freedom. 
        In addition, the purely volumetric enrichment leads to symmetric positive definite
        moment matrices, guaranteeing numerical stability and simplifying
        large-scale implementations.
\end{itemize}

The next result holds for any positive interior density and establishes that
the moment matrix arising in purely volumetric enrichment has the fundamental
symmetry and positivity properties required for unisolvence.

\begin{lemma}\label{lemma4}
The matrix $\mathcal{D}^{\mathrm{vol}}$, expressed
in the basis $\mathrm B$ of $\mathbb{S}_2(T)$ defined in~\eqref{basisfun},
is symmetric and positive definite.
\end{lemma}
\begin{proof}
By definition
\begin{equation*}
\mathcal{D}^{\mathrm{vol}}=\left[\left\langle b_m,b_n\right\rangle_{T,\Omega}\right]_{1\le m,n\le6},
\end{equation*}
where $b_m\in\mathrm B$ and
$\langle\cdot,\cdot\rangle_{T,\Omega}$ is the weighted
inner product introduced in~\eqref{innprd}.
Since this is a real inner product, $\mathcal{D}^{\mathrm{vol}}$ is
symmetric. Let $\B c=\left[c_1,\dots,c_6\right]^\top\in\mathbb R^6$ and define
\begin{equation}\label{qpols}
    q(\B x)=\sum_{m=1}^6 c_m b_m(\B x)\in \mathbb{S}_2(T).
\end{equation}
Then
\begin{equation*}
    \B c^\top \mathcal{D}^{\mathrm{vol}} \B c
   = \left\langle q,q\right\rangle_{T,\Omega}
   = \int_T \left\lvert q(\B x)\right\rvert^2\Omega(\B x)d\B x .
\end{equation*}
Since $\Omega$ is positive on the interior of $T$, the integral
vanishes only when $q=0$. By \eqref{qpols} this implies $c_1=\cdots=c_6=0$, proving that
$\mathcal{D}^{\mathrm{vol}}$ is symmetric and positive definite.
\end{proof}

\subsection{Edge-Face enrichment strategy}
We finally introduce a third enrichment strategy on tetrahedral meshes,
referred to as the \emph{edge-face strategy}.
This construction retains the classical four face averages, thereby guaranteeing
exact mass conservation on each face, while enriching the quadratic component
of $\mathbb{P}_2(T)$ through probabilistic moments supported on the
one-dimensional edges of the tetrahedron. 
For $1\le i<j\le4$ let $\epsilon_{ij}$ denote the edge connecting
$\B{v}_i$ and $\B{v}_j$, with parametrization
\begin{equation*}
    x_{ij}(t) = (1-t)\B v_i + t\B v_j, \quad t\in[0,1].
\end{equation*}
Using~\eqref{propisp}, the barycentric coordinates of a point on $\epsilon_{ij}$ satisfy
\begin{equation*}
\lambda_k\bigl(x_{ij}(t)\bigr)=
\begin{cases}
1 - t, & k = i,\\[4pt]
t, & k = j,\\[4pt]
0, & k \notin \{i,j\}.
\end{cases}
\end{equation*}
This key observation allows the construction of edge-supported degrees of
freedom in close analogy with the face and volume moments of the previous
sections. For each edge $\epsilon_{ij}$ we prescribe a one-dimensional probability density $\omega_{ij}(t),$ $t\in[0,1]$ and select a quadratic polynomial $q_{ij}\in\mathbb{P}_2([0,1])$ orthogonal to all
affine functions with respect to $\omega_{ij}$.
The associated \emph{edge moment} is defined by
\begin{equation*}
   \mathcal{L}_{ij} : p \in \mathbb{P}_2(T) \mapsto \int_0^1 p\left(x_{ij}(t)\right)q_{ij}(t)\omega_{ij}(t)dt\in \mathbb R,
\quad 1\le i<j\le4.
\end{equation*}
Collecting these functionals, we define the edge-face enriched triple
\begin{equation*}
\mathcal{AF}^{\mathrm{ef}}
=\left(T,\mathbb{P}_2(T),\Sigma^{\mathrm{ef}}\right),
\end{equation*}
where $\Sigma^{\mathrm{ef}}
= \left\{\mathcal{I}_j \, :\, j=1,2,3,4\right\}\cup\left\{\mathcal{L}_{ij}\,:\,1\le i<j\le4\right\}.$
In analogy with the previous approaches, the restriction of
$\Sigma^{\mathrm{ef}}$ to the pure quadratic subspace $\mathbb{S}_2(T)$
is represented by the matrix associated with the linear map
\begin{equation*}
   \mathcal{D}^{\mathrm{ef}} : p \in \mathbb{S}_2(T) \mapsto \left[\mathcal{L}_{12}(p),\mathcal{L}_{13}(p), \mathcal{L}_{14}(p),\mathcal{L}_{23}(p),\mathcal{L}_{24}(p),\mathcal{L}_{34}(p)\right]^{\top}\in\mathbb{R}^6,
\end{equation*}
whose rank determines the unisolvence of the enriched triple.
The following result provides the unisolvence criterion for the edge-face construction.

\begin{theorem}[Unisolvence of the edge-face strategy]
The triple $\mathcal{AF}^{\mathrm{ef}}$ is unisolvent if and only if
\begin{equation*}
    \int_0^1 t(1-t)q_{ij}(t)\omega_{ij}(t)dt \neq 0, \quad 1\le i<j\le 4.
\end{equation*}
\end{theorem}

\begin{proof}
By Lemma~\ref{lem:invert-M}, it suffices to prove the unisolvence of the
enriched functionals on the quadratic subspace $\mathbb{S}_2(T)$.
Using the standard basis $\mathrm{B}$ defined in~\eqref{basisfun},
each functional $\mathcal{L}_{ij}$ acts diagonally, i.e.,
\begin{equation}
\mathcal{L}_{ij}\left(\lambda_\ell\lambda_s\right)=
\begin{cases}
\int_{0}^{1} t(1-t)q_{ij}(t)\omega_{ij}(t)dt,
  & \text{if }\{\ell,s\}=\{i,j\},\\[6pt]
0, & \text{if }\{\ell,s\}\neq\{i,j\}.
\end{cases}
\end{equation}
Consequently, the operator $\mathcal{D}^{\mathrm{ef}}$ is represented by a diagonal matrix whose diagonal entries are precisely these
integrals.  
The requirement that all such integrals are nonzero is therefore equivalent to 
$\operatorname{rank}\left(\mathcal{D}^{\mathrm{ef}}\right)=6,$
which proves the claimed unisolvence.
\end{proof}

In analogy with~\eqref{pi12}, we can therefore define the quadratic reconstruction operator
$\pi^{\mathrm{ef}}_{2}$ associated with the enriched triple $\mathcal{AF}^{\mathrm{ef}}$.

\begin{example}[Beta edge densities]
A convenient family of edge densities is provided by the Beta law
\begin{equation}\label{eq:beta-density}
  \omega_{ij}^{(\zeta,\nu)}(t)
    = \frac{t^{\zeta-1}(1-t)^{\nu-1}}{B(\zeta,\nu)},
    \quad \zeta,\nu>0,
\end{equation}
where $B(\zeta,\nu)$ is the Euler Beta function; see
Fig.~\ref{densfig}.
  \begin{figure}
      \centering
\includegraphics[width=0.49\linewidth]{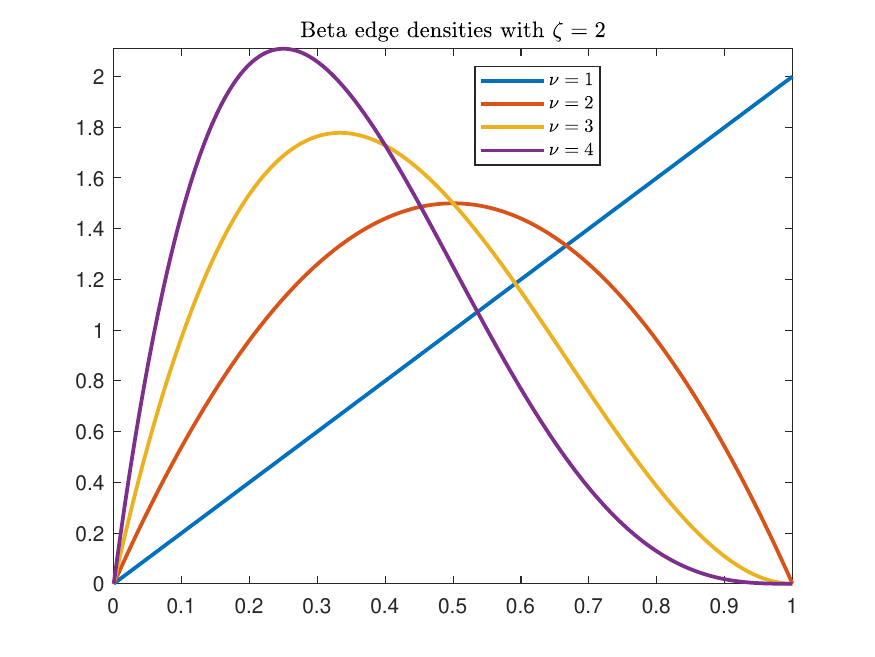}
\includegraphics[width=0.49\linewidth]{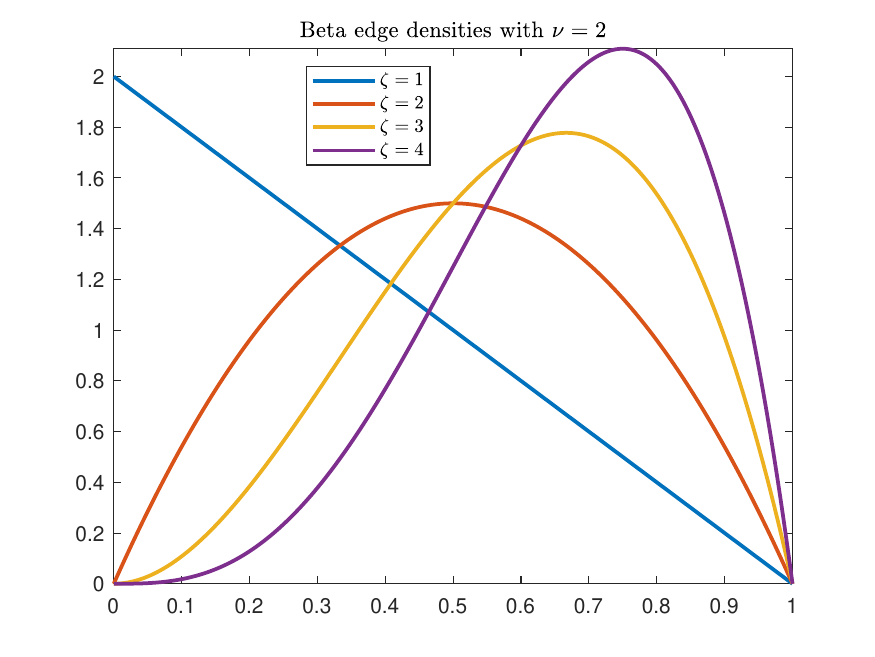}
\caption{Densities $\omega_{ij}^{(\zeta,\nu)}$ on $t\in[0,1]$:
   dependence on $\nu$ for fixed $\zeta$ (left) and on $\zeta$ for fixed $\nu$ (right).}
      \label{densfig}
  \end{figure}
In this case we select the quadratic polynomial
\begin{equation}\label{eq:q-zeta-nu}
  q_{\zeta,\nu}(t)
   = (\zeta+\nu+2)(\zeta+\nu+3)t^{2}
     -2(\zeta+1)(\zeta+\nu+2)t
     +(\zeta+1)(\zeta+2),
\end{equation}
which is, up to a constant factor, the Jacobi polynomial
$p_{J,2}^{(\zeta-1,\nu-1)}(2t-1)$.  
As a consequence, all the required properties are satisfied~\cite{szego1975orthogonal}.
\end{example}

\begin{remark}
We observe that
\begin{itemize}
  \item the edge-face strategy preserves face averages while enriching
        the quadratic space through six inexpensive one-dimensional moments;
  \item the unisolvence requirement reduces to a single nonvanishing
        moment on each edge, which holds for generic Beta parameters
        $\zeta,\nu>0$;
  \item the explicit formula for $q_{\zeta,\nu}$ enables a straightforward
        numerical implementation;
  \item compared with the face-volume and purely volumetric families,
        this strategy provides a directional enrichment
        particularly well suited to edge-aligned features.
\end{itemize}
\end{remark}

\section{Dirichlet densities}
\label{ex3}
This section illustrates in detail how the abstract enrichment strategies of
Section~\ref{sec2} can be realized with concrete parametric probability
densities.  
We focus on Dirichlet families, which provide a flexible yet analytically
tractable class of weights.  
Two complementary cases are addressed:
\begin{itemize}
   \item the \emph{Dirichlet face-volume enrichment}, where face and volume
         densities are independently parametrized by positive shape parameters
         $\alpha$ and $\beta$;
   \item the \emph{Dirichlet purely-volumetric enrichment}, where a single
         volumetric Dirichlet law, possibly combined with convex blends, governs
         all higher-order moments.
\end{itemize}

\subsection{Dirichlet face–volume enrichment}
We first consider the case where Dirichlet densities with parameters
$\alpha,\beta>0$ are prescribed on each face of the tetrahedron and in the
volume. Using the same notation as in the previous section, define
\begin{equation}\label{densins}
\omega_j^{(\alpha)}=\frac{\Gamma(3\alpha)}{\Gamma(\alpha)^3}
\mu_{1,j}^{\alpha-1}\mu_{2,j}^{\alpha-1}\mu_{3,j}^{\alpha-1}, \quad 
\Omega^{(\beta)}=\frac{\Gamma(4\beta)}{\Gamma(\beta)^4}
\prod_{i=1}^4 \lambda_i^{\beta-1}, \quad j=1,2,3,4.
\end{equation}
The parameter $\alpha$ controls the concentration on each face $F_j$, whereas
$\beta$ controls the concentration in the interior; see Fig.~\ref{fig:densities}.
   \begin{figure}
      \centering
\includegraphics[width=0.49\linewidth]{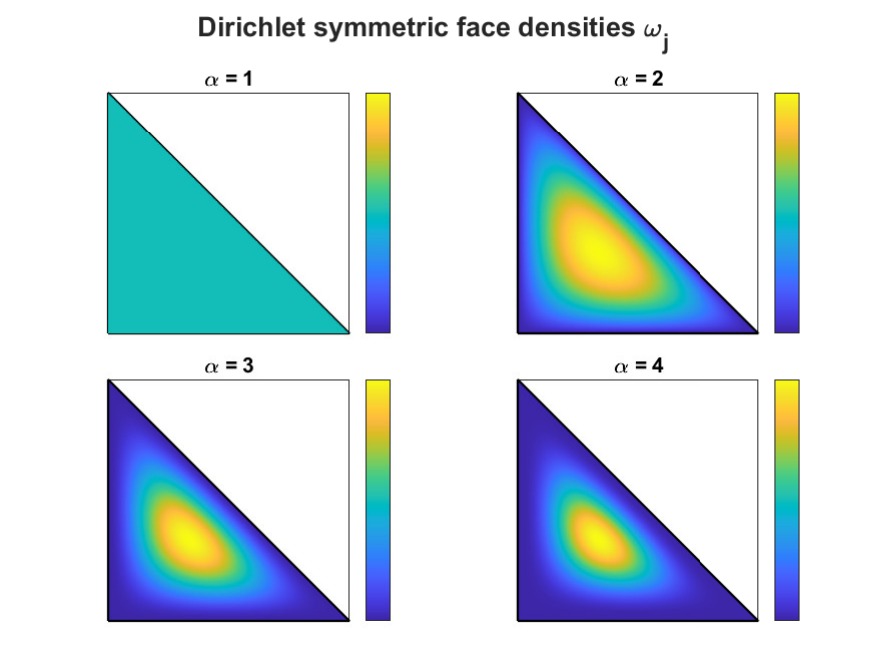}
\includegraphics[width=0.49\linewidth]{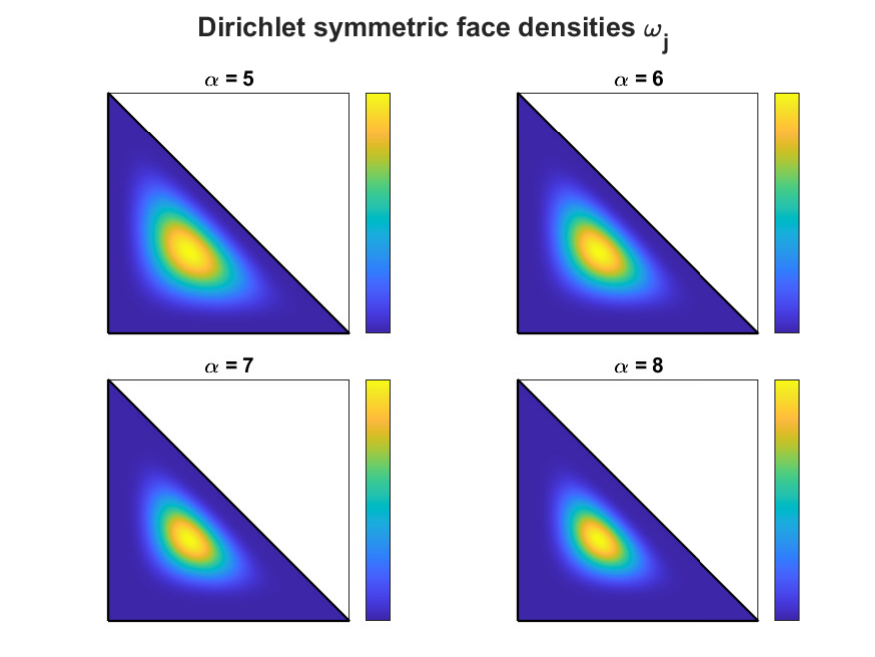}
\caption{Dirichlet face densities $\omega_j$ on the reference tetrahedron for different values of the parameter $\alpha$.}
      \label{fig:densities}
  \end{figure}
On each face we employ the quadratic polynomial
\begin{equation*}
      q_j^{(\alpha)}=\mu_{1,j}^2+\mu_{2,j}^2+\mu_{3,j}^2-c_\alpha, \quad
 c_\alpha=\frac{\alpha+1}{3\alpha+1}.
\end{equation*}
A direct computation shows that $q_j^{(\alpha)}$ is orthogonal to all affine
functions on $F_j$ with respect to the density $\omega_j^{(\alpha)}$, namely
\begin{equation*}
\int_{F_j} q_j^{(\alpha)}(\B x)\omega_j^{(\alpha)}(\B x)d\B x=0,\quad
\int_{F_j} q_j^{(\alpha)}(\B x)\mu_{r,j}(\B x)\omega_j^{(\alpha)}(\B x)d\B x=0,\quad r=1,2,3.
\end{equation*}
In the volume we consider the quadratic polynomials
\begin{equation*}
    \rho_1^{(\beta)}=\lambda_1\lambda_2+h_{\beta}-k_{\beta}\left(\lambda_1+\lambda_2\right), \quad  
   \rho_2^{(\beta)}=\lambda_1\lambda_3+h_{\beta}-k_{\beta}\left(\lambda_1+\lambda_3\right),
\end{equation*}
with parameters
\begin{equation*}
    h_{\beta}=\frac{\beta^2}{2(2\beta+1)(4\beta+1)}, \quad
    k_{\beta}=\frac{\beta}{2(2\beta+1)}.
\end{equation*}
A direct computation further shows that
\begin{equation*}
\frac{1}{\lvert T\rvert}\int_T \rho_k^{(\beta)}(\B x)\Omega^{(\beta)}(\B x)d\B x=0,\quad
\frac{1}{\lvert T\rvert}\int_T \rho_k^{(\beta)}(\B x)\lambda_r(\B x)\Omega^{(\beta)}(\B x)d\B x=0,
\end{equation*}
so that $\rho_1^{(\beta)}$ and $\rho_2^{(\beta)}$ are orthogonal to all affine
functions on $T$ with respect to the density $\Omega^{(\beta)}$. In the following we set
\begin{eqnarray*}
d_{\alpha}= -\frac{2\alpha}{9(3\alpha+1)^2(3\alpha+2)}, \quad v_{\beta}=
\frac{\beta\left(5\beta^{2}+5\beta+1\right)}  {8 \left(1 + 2 \beta\right)^2 \left(1 + 4 \beta\right)^2 \left(3 + 4 \beta\right)},
\end{eqnarray*}
\begin{equation*}
    w_{\beta}=
\frac{\beta^{3}}     {8 \left(1 + 2 \beta\right)^2 \left(1 + 4 \beta\right)^2 \left(3 + 4 \beta\right)}, \quad u_{\beta}=
-\frac{\beta^{2}}
{16 \left(1 + 2 \beta\right) \left(1 + 4 \beta\right)^2 \left(3 + 4 \beta\right)}.
\end{equation*}
Using the basis $\mathrm{B}$ defined in~\eqref{basisfun}, the matrix
$\mathcal{D}_{\alpha,\beta}^{\mathrm{fv}}$ takes the form
\begin{equation}\label{Dalphabeta}
\mathcal{D}_{\alpha,\beta}^{\mathrm{fv}}=
\begin{pmatrix}
0 & 0 & 0 & d_{\alpha} & d_{\alpha} & d_{\alpha} \\
0 & d_{\alpha} & d_{\alpha} & 0 & 0 & d_{\alpha} \\
d_{\alpha} & 0 & d_{\alpha} & 0 & d_{\alpha} & 0 \\
d_{\alpha} & d_{\alpha} & 0 & d_{\alpha} & 0 & 0 \\
v_{\beta} & u_{\beta} & u_{\beta} & u_{\beta} & u_{\beta} & w_{\beta} \\
u_{\beta} & v_{\beta} & u_{\beta} & u_{\beta} & w_{\beta} & u_{\beta}
\end{pmatrix}.
\end{equation}
A symbolic factorization yields
\begin{equation*}
\det \left(\mathcal D_{\alpha,\beta}^{\mathrm{fv}}\right)=
\frac{\alpha^4\beta^2}{2(3(3\alpha+1))^8(3\alpha+2)^4(2\beta+1)^2(4\beta+1)^2(4\beta+3)^2},
\end{equation*}
which is strictly positive for all $\alpha,\beta>0$. Therefore $\operatorname{rank}\left(\mathcal D_{\alpha,\beta}^{\mathrm{fv}}\right)=6$, for any $\alpha,\beta>0.$

\begin{remark}
The Dirichlet construction includes the uniform density of
Example~\ref{ex1} as a special case.
Indeed, setting $\alpha=\beta=1$ reduces the Dirichlet face and volume
densities to the uniform ones, so that all orthogonality relations and
associated functionals coincide with those of Example~\ref{ex1}.
Thus the densities in~\eqref{densins} provide a natural
two-parameter extension of the uniform setting.
\end{remark}
For these particular densities, we denote the triple~\eqref{AFfv} by
\begin{equation*}
       \mathcal{AF}^{\mathrm{fv}}_{\alpha,\beta}
      = \left(T,\mathbb{P}_2(T),\Sigma_{\alpha,\beta}^{\mathrm{fv}}\right).
\end{equation*}
In this case, the structure of the matrix $\mathcal{H}_{\alpha,\beta}$,
which determines the basis functions, becomes easier to handle,
since several of its blocks can be computed explicitly.
To establish this fact, we first prove a few auxiliary lemmas.

\begin{lemma}\label{lemma11ab}
For $\alpha>0$ and for any $j=1,2,3,4$, the following identities hold 
\begin{equation*}
    \int_{F_j}\mu_{r,j}(\B x)\omega_j^{(\alpha)}(\B x) d\B x= \frac{1}{3}, \quad 
\int_{F_j}\mu_{r,j}(\B x)\mu_{s,j}(\B x)\omega_j^{(\alpha)}(\B x) d\B x = \frac{\alpha}{3(3\alpha+1)},
\end{equation*}
where $\omega_j^{(\alpha)}$ is the face Dirichlet density defined in~\eqref{densins}.
\end{lemma}
\begin{proof}
We set $\widehat\omega_j^{(\alpha)}=\mu_{1,j}^{\alpha-1}\mu_{2,j}^{\alpha-1}\mu_{3,j}^{\alpha-1}, $
so that
\begin{equation}\label{omegane-ab}
 \omega_j^{(\alpha)}=\frac{\widehat\omega_j^{(\alpha)}}{\int_{F_j}\widehat\omega_j^{(\alpha)}(\B x)d\B x}.
\end{equation}
For any $\alpha_1,\alpha_2,\alpha_3>-1$, Lemma~\ref{lem1} yields
\begin{equation}\label{eq:face-moment-ab}
\int_{F_j}\mu_{1,j}^{\alpha_1}(\B x)\mu_{2,j}^{\alpha_2}(\B x)\mu_{3,j}^{\alpha_3}(\B x)d\B x
=2\left|F_j\right|
\frac{\Gamma\left(\alpha_1+1\right)\Gamma\left(\alpha_2+1\right)\Gamma\left(\alpha_3+1\right)}
     {\Gamma\left(3+\alpha_1+\alpha_2+\alpha_3\right)}.
\end{equation}
Using~\eqref{omegane-ab} and~\eqref{eq:face-moment-ab}, and recalling $\Gamma(\beta+1)=\beta\Gamma(\beta)$, we obtain
\begin{equation*}
\int_{F_j}\mu_{1,j}(\B x)\omega_j^{(\alpha)}(\B x) d\B x
=\frac{\int_{F_j}\mu_{1,j}^{\alpha}(\B x)\mu_{2,j}^{\alpha-1}(\B x)\mu_{3,j}^{\alpha-1}(\B x) d\B x}{\int_{F_j}\mu_{1,j}^{\alpha-1}(\B x)\mu_{2,j}^{\alpha-1}(\B x)\mu_{3,j}^{\alpha-1}(\B x) d\B x}
=\frac{\Gamma(\alpha+1)}{\Gamma(\alpha)}\frac{\Gamma(3\alpha)}{\Gamma(3\alpha+1)}
=\frac{1}{3}.
\end{equation*}
Similarly,
\begin{equation*}
 \int_{F_j}\mu_{1,j}(\B x)\mu_{2,j}(\B x)\omega_j^{(\alpha)}(\B x) d\B x
=\frac{\int_{F_j}\mu_{1,j}^{\alpha}(\B x)\mu_{2,j}^{\alpha}(\B x)\mu_{3,j}^{\alpha-1}(\B x) d\B x}{\int_{F_j}\mu_{1,j}^{\alpha-1}(\B x)\mu_{2,j}^{\alpha-1}(\B x)\mu_{3,j}^{\alpha-1}(\B x) d\B x}
=\frac{\alpha}{3(3\alpha+1)}.
\end{equation*}
By symmetry the same formulas hold for every $r$ and every pair $r\neq s$, which completes the proof.
\end{proof}

The matrix  $\mathcal{H}_{\alpha,\beta}$ can be written as 
\begin{equation}\label{matrixH_blocks_abd}
\mathcal{H}_{\alpha,\beta}=
\left[
\begin{array}{c|c}
\mathcal{M} & \mathcal{N}_{\alpha}\\
\hline
\B 0 & \mathcal{D}_{\alpha,\beta}^{\mathrm{fv}}
\end{array}
\right]\in\mathbb{R}^{10\times10},
\end{equation}
where
\begin{equation}\label{matA_ab}
\mathcal{M}=\frac{1}{3}\left[
\begin{array}{cccc}
0 & 1 & 1 & 1\\
1 & 0 & 1 & 1\\
1 & 1 & 0 & 1\\
1 & 1 & 1 & 0
\end{array}
\right], \quad \mathcal{N}_{\alpha}=\frac{\alpha}{3(3\alpha+1)}\left[
\begin{array}{cccccc}
0 & 0 & 0 & 1 & 1 & 1\\
0 & 1 & 1 & 0 & 0 & 1 \\
1 & 0 & 1 & 0& 1 & 0\\
1 & 1 & 0 & 1 & 0 & 0
\end{array}
\right].
\end{equation}

\subsection{Dirichlet volumetric enrichment}

We now consider a purely volumetric realization of the enrichment strategy,
where all higher-order moments are supported inside the tetrahedron.
Let $\gamma>0$ and define the Dirichlet interior density
\begin{equation*}
    \Omega^{(\gamma)}=
\frac{\Gamma(4\gamma)}{\Gamma(\gamma)^4}
\prod_{i=1}^4\lambda_i^{\gamma-1}
\end{equation*}
on the tetrahedron $T$.  
For each pair $1\le i<j\le4$ we introduce the quadratic polynomial
\begin{equation}\label{spasasasasasa}
    \psi_{ij}(\B x)
   =\lambda_i(\B x)\lambda_j(\B x)
     + h_\gamma - k_\gamma\left(\lambda_i(\B x)+\lambda_j(\B x)\right),
   \quad 1\le i<j\le4,
\end{equation}
where 
\begin{equation*}
h_\gamma=\frac{\gamma^2}{2(2\gamma+1)(4\gamma+1)}, 
\quad
k_\gamma=\frac{\gamma}{2(2\gamma+1)}.    
\end{equation*}
By construction,
\begin{equation*}
    \left\langle \psi_{ij},\phi\right\rangle_{T,\Omega^{(\gamma)}}=0,
\quad
\forall \phi\in\mathbb{P}_{1}(T).
\end{equation*}
The associated volumetric moments are
\begin{equation*}
    \mathcal{F}^{(\gamma)}_{ij}(p)
   =\int_T \psi_{ij}(\B x)p(\B x)\Omega^{(\gamma)}(\B x)d\B x
\end{equation*}
and give rise to the set of degrees of freedom
$$\Sigma^{\mathrm{vol}}_{\gamma}
   =\left\{\mathcal{I}_{j}\, :\,j=1,2,3,4\right\}\cup\left\{\mathcal{F}^{(\gamma)}_{ij}\, : \, 1\le i<j\le4\right\}. $$
Using the basis $\mathrm{B}$ of $\mathbb{S}_2(T)$ defined in~\eqref{basisfun},
the restriction of these functionals to $\mathbb{S}_2(T)$
is represented by the symmetric matrix
\begin{equation*}
\mathcal{D}^{\mathrm{vol}}_{\gamma}=
\begin{pmatrix}
s_{\gamma} & t_{\gamma} & t_{\gamma} & t_{\gamma} & t_{\gamma} & z_{\gamma}\\
t_{\gamma} & s_{\gamma} & t_{\gamma} & t_{\gamma} & z_{\gamma} & t_{\gamma}\\
t_{\gamma} & t_{\gamma} & s_{\gamma} & z_{\gamma} & t_{\gamma} & t_{\gamma}\\
t_{\gamma} & t_{\gamma} & z_{\gamma} & s_{\gamma} & t_{\gamma} & t_{\gamma}\\
t_{\gamma} & z_{\gamma} & t_{\gamma} & t_{\gamma} & s_{\gamma} & t_{\gamma}\\
z_{\gamma} & t_{\gamma} & t_{\gamma} & t_{\gamma} & t_{\gamma} & s_{\gamma}
\end{pmatrix},
\end{equation*}
where
\begin{equation*}
s_{\gamma}=\frac{\gamma(5\gamma^2+5\gamma+1)}{8(1+2\gamma)^2(1+4\gamma)^2(3+4\gamma)},\quad
t_{\gamma}=-\frac{\gamma^2}{16(1+2\gamma)(1+4\gamma)^2(3+4\gamma)},
\end{equation*}
\begin{equation*}
    z_{\gamma}=\frac{\gamma^3}{8(1+2\gamma)^2(1+4\gamma)^2(3+4\gamma)}.
\end{equation*}
A straightforward symbolic computation shows that
\begin{equation}\label{dvolgamma}
\det\left(\mathcal{D}^{\mathrm{vol}}_{\gamma}\right)
=\frac{\gamma^{6}(\gamma+1)^4}{2^{18}\,(2\gamma+1)^{9}\,(4\gamma+1)^{7}\,(4\gamma+3)^{6}},
\end{equation}
which is strictly positive for all $\gamma>0$. 
Hence, by Theorem~\ref{thrmimps1}, the purely volumetric Dirichlet-enriched triple $\mathcal{AF}^{\mathrm{vol}}_{\gamma}$ is unisolvent for every $\gamma>0$.

\begin{remark}
As shown in Lemma~\ref{lemma4}, the matrix $\mathcal{D}^{\mathrm{vol}}_\gamma$ is symmetric and positive definite for every $\gamma>0$; hence its determinant is automatically  positive.  
\end{remark}

A natural extension of the Dirichlet volumetric enrichment is to interpolate
between the uniform interior density and a Dirichlet interior density.
Let $0\le\theta\le 1$ and fix $\gamma>0$.
We introduce the convex combination of the uniform density
$\Omega^{(1)}=|T|^{-1}$ and the Dirichlet density $\Omega^{(\gamma)}$,
\begin{equation}\label{omegamix}
    \Omega^{(\theta,\gamma)}(\B x)
=\theta\Omega^{(1)}(\B x)
     +(1-\theta)\Omega^{(\gamma)}(\B x).
\end{equation}
Clearly $\Omega^{(\theta,\gamma)}$ is a probability density on $T$ for every
$\theta\in[0,1]$ and $\gamma>0$, though it is not in general of Dirichlet type.
Using $\Omega^{(\theta,\gamma)}$ in place of $\Omega^{(\gamma)}$, we define the
purely volumetric moments
\begin{equation*}
\mathcal{F}_{ij}^{(\theta,\gamma)}(p)
   =\int_{T}\psi_{ij}(\B x)p(\B x)\Omega^{(\theta,\gamma)}(\B x)d\B x,
   \quad 1\le i<j\le4,    
\end{equation*}
where $\psi_{ij}$, $1\le i<j\le4$ are defined in~\eqref{spasasasasasa}. 
The corresponding volumetric enriched triple is $\mathcal{AF}^{\mathrm{vol}}_{\theta,\gamma}
   =\left(T,\mathbb{P}_{2}(T),
            \Sigma^{\mathrm{vol}}_{\theta,\gamma}\right)$,
with the set of degrees of freedom 
\begin{equation*}
\Sigma^{\mathrm{vol}}_{\theta,\gamma}
   =\left\{\mathcal{I}_j\, : \,j=1,2,3,4\right\}\cup
      \left\{\mathcal{F}_{ij}^{(\theta,\gamma)}\, : \, 1\le i<j\le4\right\}.
\end{equation*}
We denote by $\mathcal{D}^{\mathrm{vol}}_{\theta,\gamma}$ the restriction of
$\Sigma^{\mathrm{vol}}_{\theta,\gamma}$ to the pure quadratic subspace
$\mathbb{S}_{2}(T)$.
By the convex combination~\eqref{omegamix} defining $\Omega^{(\theta,\gamma)}$,
the corresponding moment matrix is itself a convex combination
\begin{equation*}
\mathcal{D}^{\mathrm{vol}}_{\theta,\gamma}
=\theta\mathcal{D}^{\mathrm{vol}}_{1}
     +(1-\theta)\mathcal{D}^{\mathrm{vol}}_{\gamma},
\end{equation*}
where $\mathcal{D}^{\mathrm{vol}}_{1}$ and
$\mathcal{D}^{\mathrm{vol}}_{\gamma}$ are the volumetric matrices
associated with the uniform and Dirichlet densities, respectively.
Since, by Lemma~\ref{lemma4}, both $\mathcal{D}^{\mathrm{vol}}_{1}$ and
$\mathcal{D}^{\mathrm{vol}}_{\gamma}$ are symmetric and positive definite,
their convex combination
$\mathcal{D}^{\mathrm{vol}}_{\theta,\gamma}$ is also symmetric and positive
definite for every $\theta\in[0,1]$ and $\gamma>0$.
In particular, $\det\left(\mathcal{D}^{\mathrm{vol}}_{\theta,\gamma}\right) > 0$, 
and by Theorem~\ref{thrmimps1} the enriched triple
$\mathcal{AF}^{\mathrm{vol}}_{\theta,\gamma}$ is unisolvent. 

 \begin{remark}
More generally, the interior density $\Omega^{(\theta,\gamma)}$ can be replaced
by any convex combination of two (or more) Dirichlet densities, namely
\[
  \Omega^{(\theta,\gamma_1,\gamma_2)} = \theta\Omega^{(\gamma_1)} + (1-\theta)\Omega^{(\gamma_2)}, \quad \theta\in[0,1], \quad \gamma_1,\gamma_2>0.
\]
Since each corresponding moment matrix is symmetric and positive definite,
the unisolvence argument applies verbatim.
\end{remark}

\begin{remark}[Limit of convex blending for face–volume enrichments]
The convex–combination argument above relies on the symmetry of
$\mathcal{D}^{\mathrm{vol}}_{\gamma}$.
For the first (face–volume) enrichment family this structural symmetry is
absent, so no analogous conclusion on the positive definiteness of a convex
blend can be drawn.
\end{remark}

\begin{remark}
In each of the three enrichment strategies, any collection of four linear
functionals that is unisolvent on $\mathbb{P}_1(T)$ could be used in place of the
specific face or volume averages $\mathcal{I}_i$, $i=1,2,3,4,$ introduced above.
This flexibility applies equally to the face-volume, purely volumetric,
and edge-face constructions.
\end{remark}

\section{Bi-Parametric tuning of Dirichlet densities}\label{sec:tuning}\label{sec3}

We address the mesh-level selection of the Dirichlet parameters
$\alpha,\beta>0$ used in the face and volume densities~\eqref{densins}.
Since $\det\left(\mathcal{D}_{\alpha,\beta}^{\mathrm{fv}}\right)>0$ for all $\alpha,\beta>0$,
the reconstruction operator $\pi^{\mathrm{fv}}_{2,\alpha,\beta}$ is well defined
over the entire parameter domain.

\medskip
\noindent
Given validation functions $\left\{f^{(r)}\right\}_{r=1}^R$ on $\mathcal Q$ and a sequence
of meshes $\left\{\mathcal T_n\right\}_{n=0}^N$, we determine the optimal pair
$\left(\alpha^\star,\beta^\star\right)$ by grid search on discrete sets
$\mathcal A,\mathcal B\subset (0,\infty)$:
\[
(\alpha^\star,\beta^\star) \in 
\arg\min_{(\alpha,\beta)\in\mathcal A\times\mathcal B}
\sum_{r=1}^R \sum_{n=0}^N
\left\| f^{(r)} - \pi^{\mathrm{fv}}_{2,\alpha,\beta}\!\left[f^{(r)};\mathcal T_n\right]\right\|_{L^1}.
\]
This offline procedure plays the role of hyperparameter tuning:
every candidate pair is admissible and the outcome can be used on all
tetrahedra and refinement levels.

\begin{algorithm}[H]
\caption{Global bi-parametric tuning of Dirichlet densities by grid search}
\label{alg:ab-grid}
\begin{algorithmic}[1]
\Require Validation functions $\{f^{(r)}\}_{r=1}^R$; meshes $\{\mathcal T_n\}_{n=0}^N$;
         candidate grids $\mathcal A,\mathcal B\subset (0,\infty)$;
         reconstruction operator $\pi^{\mathrm{fv}}_{2,\alpha,\beta}$
\Ensure Optimal pair $(\alpha^\star,\beta^\star)$
\State $\varepsilon_{\min} \gets +\infty$; \quad $(\alpha^\star,\beta^\star) \gets \text{undefined}$
\ForAll{$\alpha\in\mathcal A$}
  \ForAll{$\beta\in\mathcal B$}
    \State $\varepsilon^{\mathrm{fv}}_{2,\alpha,\beta}\gets 0$
    \For{$r=1$ to $R$}
      \For{$n=0$ to $N$}
        \State $u \gets \pi^{\mathrm{fv}}_{2,\alpha,\beta}\left[f^{(r)};\mathcal T_n\right]$
        \State $\varepsilon^{\mathrm{fv}}_{2,\alpha,\beta} \gets \varepsilon^{\mathrm{fv}}_{2,\alpha,\beta} + \left\|f^{(r)}-u\right\|_{L^1(\mathcal{Q};\mathcal T_n)}$
      \EndFor
    \EndFor
    \If{$\varepsilon^{\mathrm{fv}}_{2,\alpha,\beta} < \varepsilon_{\min}$}
       \State $\varepsilon_{\min}\gets \varepsilon^{\mathrm{fv}}_{2,\alpha,\beta}$; \quad $(\alpha^\star,\beta^\star)\gets(\alpha,\beta)$
    \EndIf
  \EndFor
\EndFor
\State \Return $(\alpha^\star,\beta^\star)$
\end{algorithmic}
\end{algorithm}

Once computed, this single pair can be applied uniformly to every tetrahedron
and mesh level, so that the quadratic reconstruction is fully specified and
reproducible.

\medskip\noindent
The same mesh-level tuning strategy applies verbatim to the second and third
enrichment families by replacing $(\alpha,\beta)$ and
$\pi^{\mathrm{fv}}_{2,\alpha,\beta}$ with the corresponding parameters and
reconstruction operators.

\section{Numerical results}\label{sec4}
We present a comprehensive numerical validation of the quadratic enriched
histopolation method proposed in this work. A key feature of the implementation is the
\emph{automatic, mesh-independent} selection of the probability densities that
characterize the enriched degrees of freedom.

We investigate three complementary enrichment strategies:
\begin{itemize}
  \item the face-volume Dirichlet enrichment with shape parameters $(\alpha,\beta)$;
  \item a purely volumetric enrichment based on a \emph{convex blend} of interior densities,
        with parameters $(\theta,\gamma)$;
  \item the Beta edge enrichment with parameters $(\zeta,\nu)$.
\end{itemize}
Rather than prescribing these parameters \emph{a priori}, we determine the optimal pairs
$\left(\alpha^\star,\beta^\star\right)$, $\left(\theta^\star,\gamma^\star\right)$, and
$\left(\zeta^\star,\nu^\star\right)$ by means of the grid search
procedures detailed in Algorithm~\ref{alg:ab-grid} and in its volumetric and edge-based
variants, respectively.

The comparison is therefore carried out among the classical histopolation
projector $\pi^{\mathrm{CH}}_{1}$, the quadratic face-volume projector 
$\pi^{\mathrm{fv}}_{2,\alpha^\star,\beta^\star}$, the quadratic volumetric
projector $\pi^{\mathrm{vol}}_{2,\theta^\star,\gamma^\star}$, and the quadratic
edge-face projector $\pi^{\mathrm{ef}}_{2,\zeta^\star,\nu^\star}$.
The validation set of functions consists of
\begin{equation*}
    f_1(x,y,z)=\sin(2\pi x)\sin(2\pi y)\sin(2\pi z), \quad f_2(x,y,z)= \sin(2 \pi xyz),
\end{equation*}
\begin{equation*}
    f_3(x,y,z)=\frac{1}{x^2+y^2+z^2+25}, \quad f_4(x,y,z)=e^{x^2+y^2+z^2},  
\end{equation*}
\begin{equation*}
    f_5(x,y,z)=\sin(x)\cos(y)e^{-z^2}, \quad f_6(x,y,z)= \left\lvert x^3\right\rvert + \left\lvert y^3\right\rvert + \left\lvert z^3\right\rvert,
\end{equation*}
\begin{equation*}
    f_7(x,y,z)=\sqrt{(x-0.5)^2 + (y-0.5)^2 + (z-0.5)^2},
\end{equation*}
\begin{eqnarray*}
f_8(x,y,z) = &&\sin\left(10\sqrt{(x-0.5)^2 + (y-0.5)^2 + (z-0.5)^2} \right)\times \\ &\times& e^{-\sqrt{(x-0.5)^2 + (y-0.5)^2 + (z-0.5)^2}}
\end{eqnarray*}
\noindent
all defined on the domain $\mathcal{Q}=[0,1]^3$. Spatial discretizations are based on a family of uniform
Delaunay-type tetrahedral meshes $\mathcal{T}_{n}=\left\{K_i\,:\, i=1,\dots,6(n-1)^3\right\},$  $n=5:5:25,$  obtained by splitting each cube of an $(n-1)^3$ Cartesian grid into six
tetrahedra; see Fig.~\ref{f1agwn}.

\begin{figure}
    \centering
\includegraphics[width=0.32\linewidth]{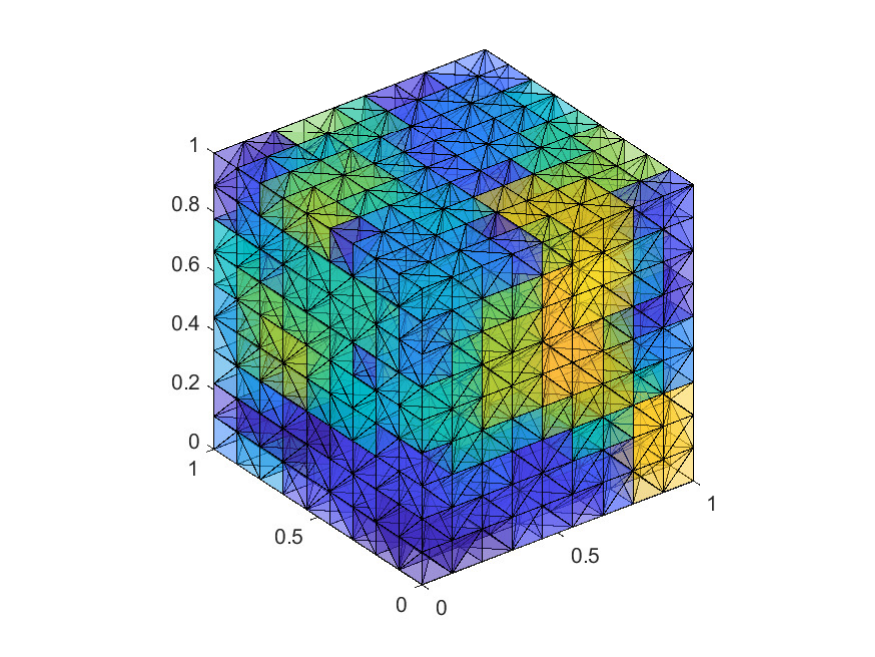}
\includegraphics[width=0.32\linewidth]{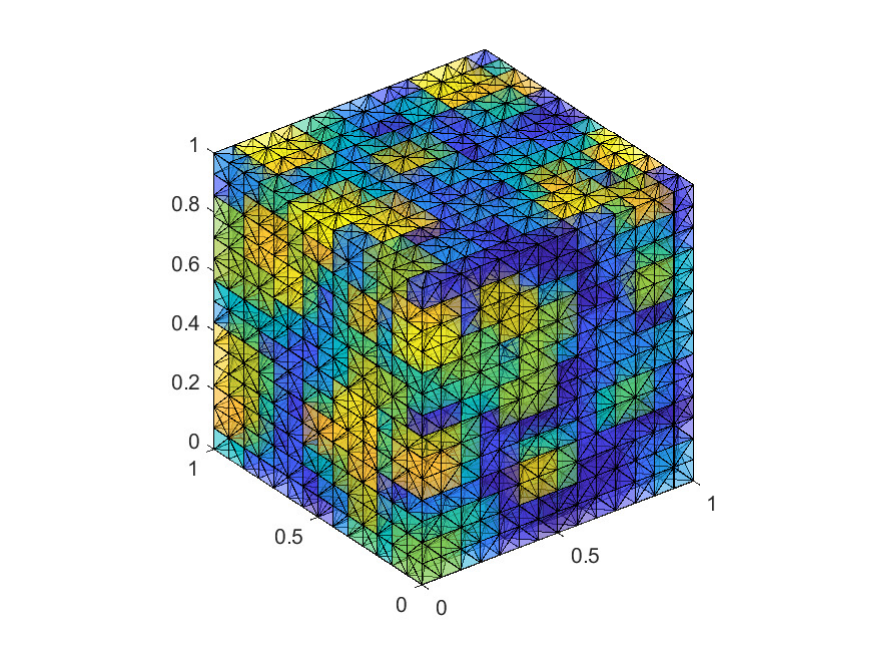}
\includegraphics[width=0.32\linewidth]{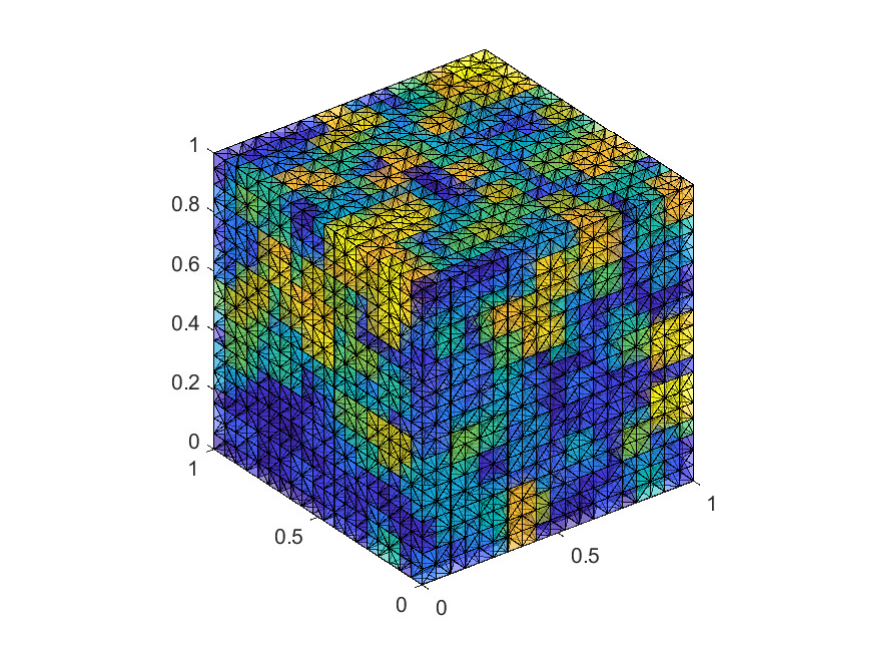}
    \caption{Regular Delaunay triangulations $\mathcal{T}_n$ for $n=10$ (left), $n=15$ (center), and $n=20$ (right).}
    \label{f1agwn}
\end{figure}
For every tetrahedron $K$ we compute the degrees of freedom of both the
standard linear and the enriched histopolation elements by means of accurate cubature
rules, ensuring that all reported errors reflect only the approximation
properties of the methods and are not polluted by quadrature noise.
All computations are performed in \textsc{MATLAB}.
Alongside the classical projector    $\pi^{\mathrm{CH}}_{1}$ we evaluate:
\begin{itemize}
  \item the quadratic face-volume projector
        $\pi^{\mathrm{fv}}_{2,\alpha^\star,\beta^\star}$;
  \item the quadratic volumetric projector$\pi^{\mathrm{vol}}_{2,\theta^\star,\gamma^\star}$;
  \item the quadratic edge-face projector $\pi^{\mathrm{ef}}_{2,\zeta^\star,\nu^\star}$.
\end{itemize}
The numerical study compares the classical scheme with all three
enriched strategies.

 \begin{figure}
      \centering
\includegraphics[width=0.49\linewidth]{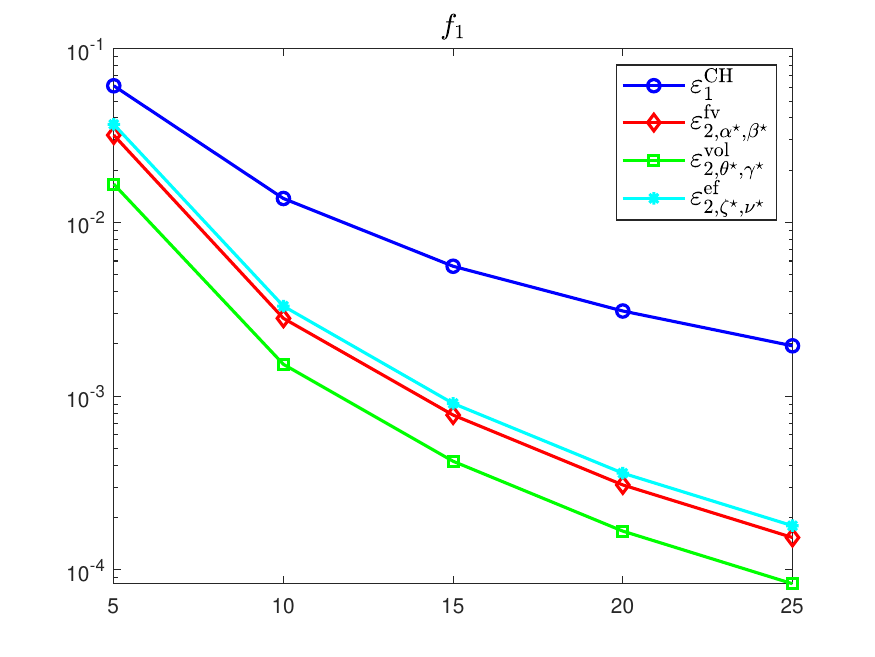}
\includegraphics[width=0.49\linewidth]{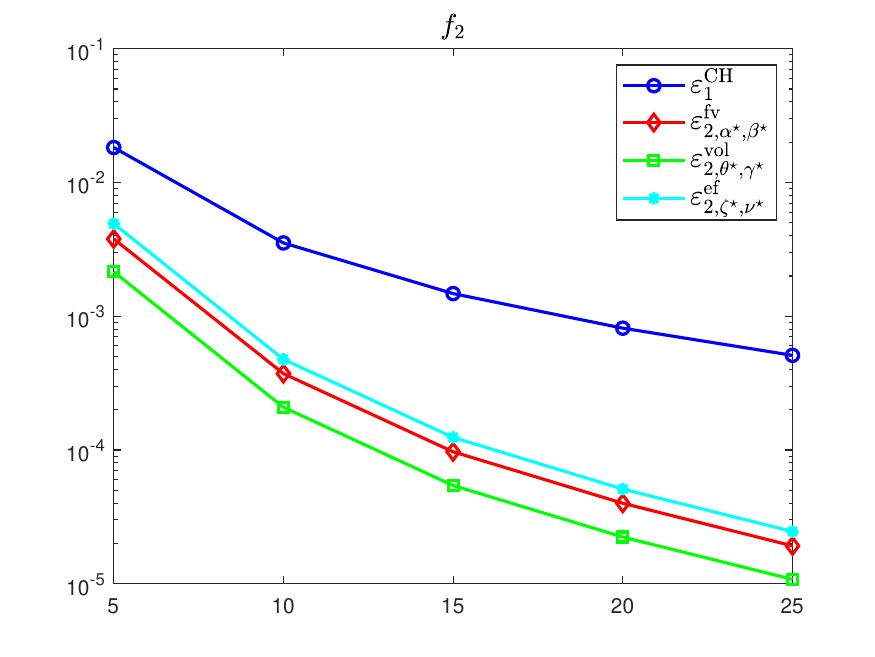}
\includegraphics[width=0.49\linewidth]{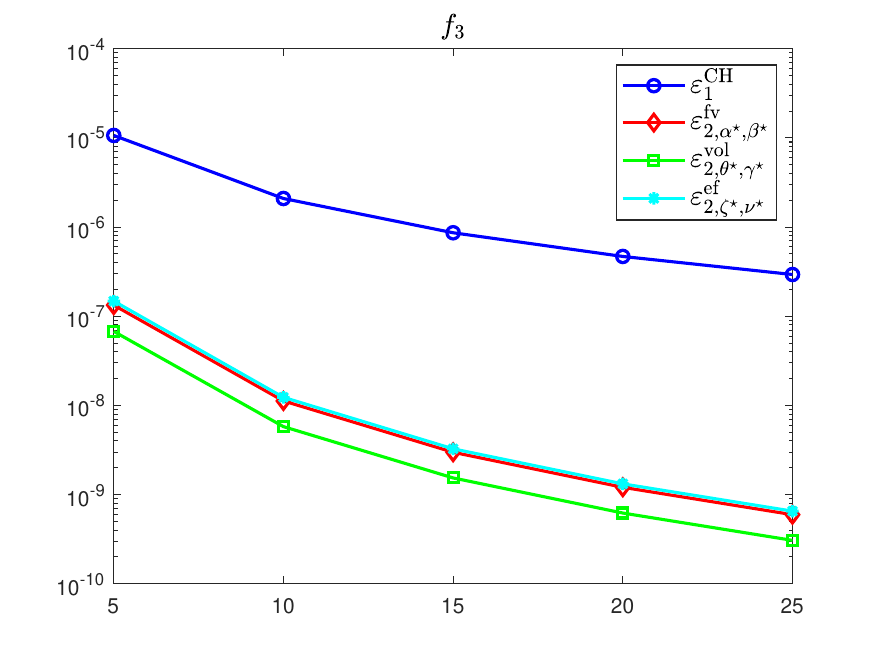}
\includegraphics[width=0.49\linewidth]{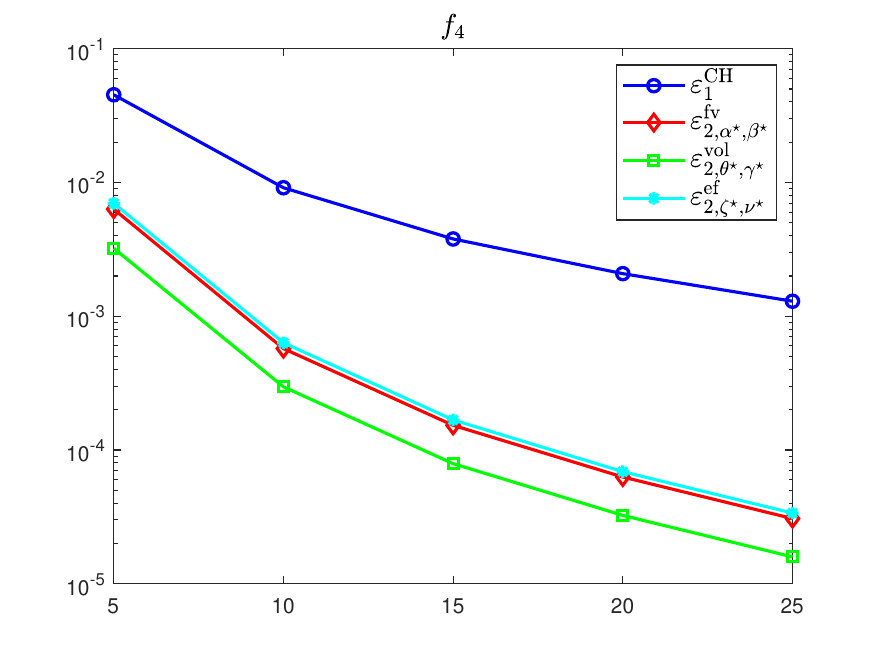}
\includegraphics[width=0.49\linewidth]{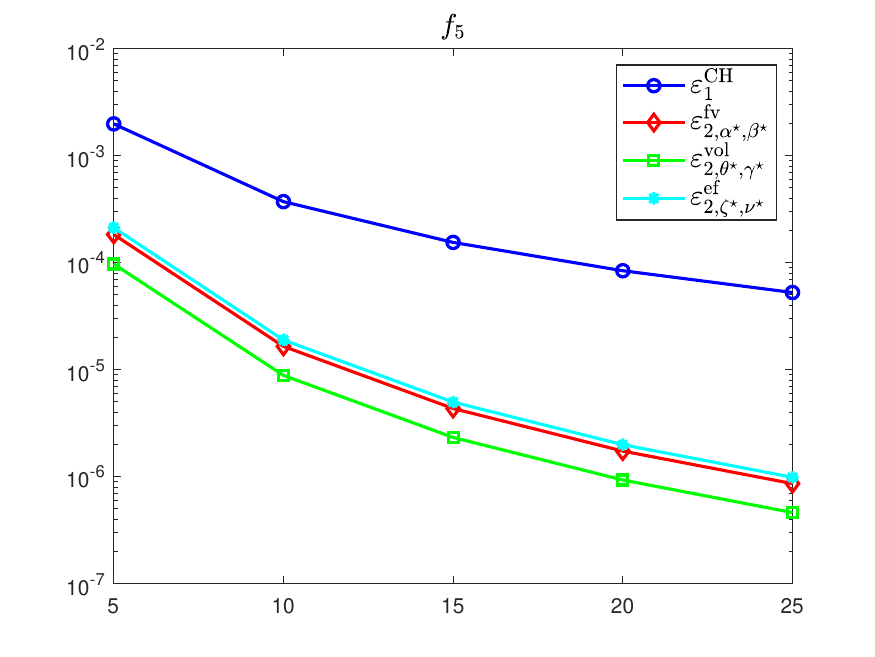}
\includegraphics[width=0.49\linewidth]{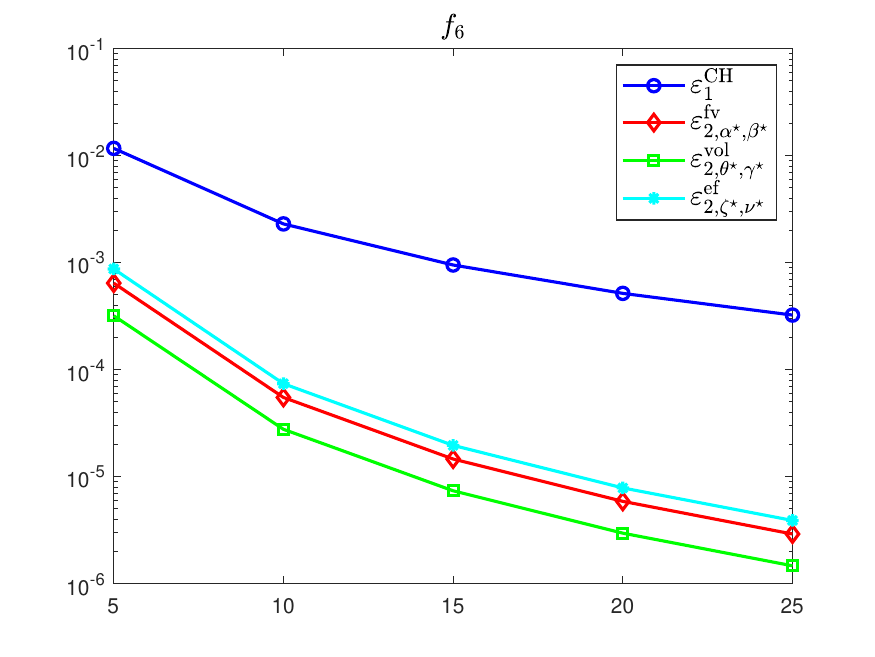}
\includegraphics[width=0.49\linewidth]{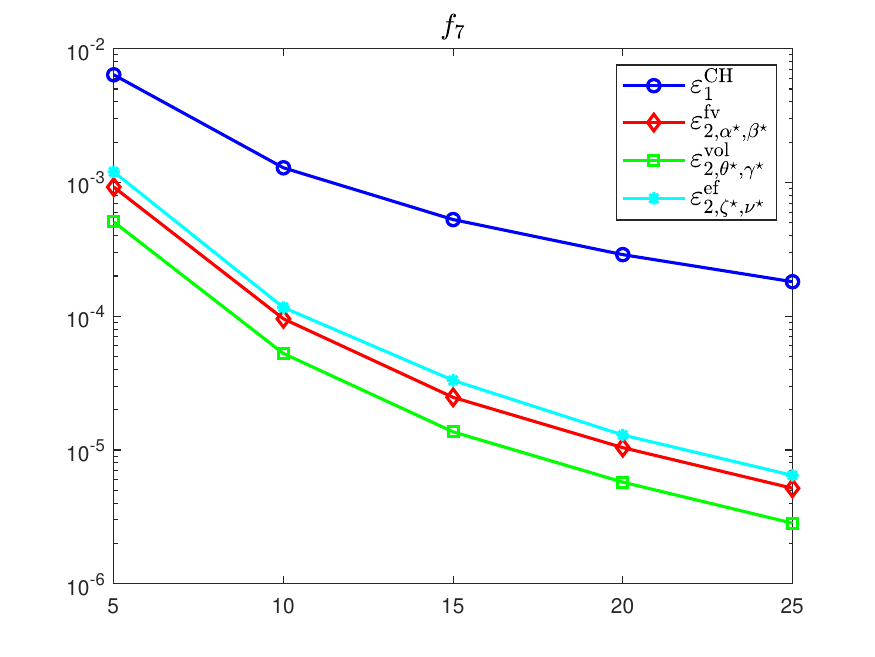}
\includegraphics[width=0.49\linewidth]{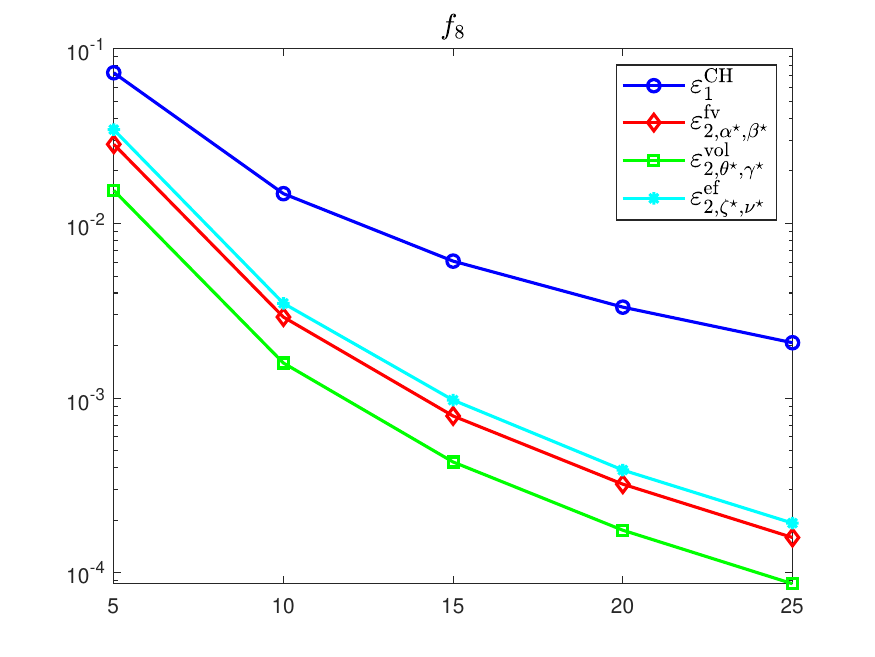}
      \caption{$L^{1}$-norm reconstruction error for $f_i$, $i=1,\dots,8$, on Delaunay triangulations $\mathcal{T}_n$, $n=5{:}5{:}25$. The comparison involves the standard histopolation (blue), the quadratic enriched method with Dirichlet weights (red), the volumetric enrichment with blended density $\Omega^{(\theta,\gamma)}$ (green), and the Beta edge enrichment strategy with parameters $(\zeta^{\star},\nu^{\star})$ (cyan). The bi-parametric tuning selects the optimal Dirichlet parameters $\alpha^{\star}=\beta^{\star}=2$, while the convex-blend tuning yields $\theta^{\star}=0.5$, $\gamma^{\star}=2$, and the edge tuning gives $\zeta^{\star}=2$, $\nu^{\star}=2$.}
      \label{fig:f12}
  \end{figure}

\noindent
Specifically, we compute the following errors in $L^1$-norm 
\begin{equation*}
    \varepsilon^{\mathrm{CH}}_1 = \left\|f-\pi^{\mathrm{CH}}_{1}[f]\right\|_{L^1},
\quad
\varepsilon^{\mathrm{fv}}_{2,\alpha^{\star}, \beta^{\star}} = \left\|f-\pi^{\mathrm{fv}}_{2,\alpha^\star,\beta^{\star}}[f]\right\|_{L^1}, 
\end{equation*}
\begin{equation*}
\varepsilon^{\mathrm{vol}}_{2,\theta^{\star},\gamma^{\star}} = \left\|f-\pi^{\mathrm{vol}}_{2,\theta^{\star},\gamma^{\star}}[f]\right\|_{L^1}, \quad \varepsilon^{\mathrm{ef}}_{2,\zeta^{\star},\nu^{\star}} = \left\|f-\pi^{\mathrm{ef}}_{2,\zeta^{\star},\nu^{\star}}[f]\right\|_{L^1}.
\end{equation*}
The results of this experiment are shown in Fig.~\ref{fig:f12}.
From this plot we observe the $L^1$ reconstruction errors for the four
projectors as the mesh is refined.
All methods improve as the mesh is refined, yet their accuracy levels
diverge progressively:
the volumetric projector 
$\pi^{\mathrm{vol}}_{2,\theta^\star,\gamma^\star}$
consistently delivers the smallest errors over all test functions,
the face-volume and edge-face projectors remain comparable to each other,
and all three enriched strategies increasingly outperform the classical
projector $\pi^{\mathrm{CH}}_{1}$ as the grid is refined.

\section{Conclusions and Future Works}\label{sec5}
We have presented three  3D weighted quadratic enrichment
strategies to enhance the accuracy of local histopolation on tetrahedral
meshes: $(i)$ a \emph{face-volume} enrichment coupling face and interior
weighted moments, $(ii)$ a \emph{purely volumetric} enrichment based solely on
interior quadratic moments, and $(iii)$ an \emph{edge-face} enrichment that
exploits Beta-type densities along edges and adjacent faces.
All three rely on weighted linear functionals defined through probability
density functions and orthogonal polynomials, extending the classical
histopolation framework while preserving its fundamental consistency.
A general unisolvence theorem guarantees well-posedness, and explicit
density families meeting its hypotheses have been identified.
Optimal parameter pairs $(\alpha^\star,\beta^\star)$,
$(\theta^\star,\gamma^\star)$, and $(\zeta^\star,\nu^\star)$ are found by a
global grid search, ensuring reproducible and problem-independent tuning.
Extensive numerical tests confirm that every enriched projector achieves
significantly higher accuracy than the classical linear histopolation:
the volumetric projector delivers the smallest overall errors, while the
face-volume and edge-face strategies exhibit comparable performance and
still markedly outperform the standard scheme. Future work will include adaptive mesh refinement and error estimation, the
integration of these enriched projectors into time-dependent PDE solvers, and
the extension of the approach to non-convex and curved polyhedral geometries.

\section*{Acknowledgments}
This research has been achieved as part of RITA \textquotedblleft Research
 ITalian network on Approximation'' and as part of the UMI group \enquote{Teoria dell'Approssimazione
 e Applicazioni}. The research was supported by GNCS–INdAM 2025 project \emph{``Polinomi, Splines e Funzioni Kernel: dall'Approssimazione Numerica al Software Open-Source''}. 
The work of F. Nudo is funded from the European Union – NextGenerationEU under the Italian National Recovery and Resilience Plan (PNRR), Mission 4, Component 2, Investment 1.2 \lq\lq Finanziamento di progetti presentati da giovani ricercatori\rq\rq,\ pursuant to MUR Decree No.~47/2025. 

\section*{Conflict of interest}
Not Applicable.

\bibliographystyle{spmpsci}
\bibliography{bibliografia}

\end{document}